\documentclass[12pt,A4paper]{article}
\usepackage[utf8]{inputenc}

\title{On the pre- and post-positional semi-random graph processes}
\author{Pu Gao \\ University of Waterloo \\ pu.gao@uwaterloo.ca \and Hidde Koerts \\ University of Waterloo \\ hkoerts@uwaterloo.ca}
\date{\today}

\usepackage{a4wide}
\usepackage{graphicx}
\usepackage{amsmath}
\usepackage{amssymb}
\usepackage{mathrsfs}
\usepackage{amsfonts}
\usepackage{enumerate}
\usepackage{dsfont}
\usepackage{amsthm}
\usepackage{fancyhdr}
\usepackage{xcolor}
\usepackage{url}
\usepackage{comment}
\usepackage[parfill]{parskip}
\usepackage[margin=3cm]{geometry}
\usepackage{geometry}
\usepackage{enumitem}
\usepackage{rotating}
\usepackage{float}
\usepackage{mathrsfs}
\usepackage{titlesec}
\usepackage{tabularray}
\usepackage{pdflscape}
\usepackage{soul}
\usepackage{caption}
\usepackage{mathtools}
\usepackage{booktabs}
\usepackage{listings}

\newcommand{\toup}{\mathrel{\nonscript\mkern-1.2mu\mkern1.2mu{\uparrow}}}

\newtheorem{theorem}{Theorem}[section]
\newtheorem{corollary}[theorem]{Corollary}
\newtheorem{conjecture}[theorem]{Conjecture}
\newtheorem{lemma}[theorem]{Lemma}
\newtheorem{proposition}[theorem]{Proposition}
\newtheorem{remark}[theorem]{Remark}

\theoremstyle{definition}
\newtheorem{definition}{Definition}[section]

\newtheorem{claim}{Claim}[section]

\makeatletter
\newcommand{\colormathbox}[3][\mathord]{%
	#1{%
		\setlength{\fboxsep}{1pt}%
		\mathpalette\color@mathbox{{#2!20}{#3}}%
	}%
}
\newcommand{\color@mathbox}[2]{%
	\color@@mathbox#1#2%
}
\newcommand{\color@@mathbox}[3]{%
	\colorbox{#2}{$#1\m@th#3$}%
}
\makeatother

\lstdefinelanguage{Maple}% 
{morekeywords={assuming,break,by,catch,description,do,done,% 
elif,else,end,error,export,fi,finally,for,from,global,if,% 
implies,in,intersect,local,minus,mod,module,next,not,od,% 
option,options,or,proc,quit,read,return,save,stop,subset,then,% 
to,try,union,use,uses,while,xor},% 
sensitive=true,% 
morecomment=[s]{(*}{*)},% 
morestring=[b]",% 
morestring=[d]"% 
}[keywords,comments,strings]% 

\DefTblrTemplate{firsthead,middlehead,lasthead}{default}{}

\newcommand\ex[1]{{\mathbb E}[#1]}

\def\eps{{\epsilon}}
\newcommand\remove[1]{{}}

\begin{document}
% \fontfamily{lmss}\selectfont

\maketitle

\begin{abstract}
    
\remove{The semi-random graph process is a single player graph game where the player is initially presented an edgeless graph with $n$ vertices. In each round, the player is offered a vertex $u$ uniformly at random and subsequently chooses a second vertex $v$ deterministically according to some strategy, and adds edge $uv$ to the graph. The objective for the player is then to ensure that the graph fulfils some specified property as fast as possible. 

We consider a variant that was recently suggested by Wormald where the player chooses the first vertex and the second vertex is chosen u.a.r. We investigate the relation between this variant and the original process. 
}

We study the semi-random graph process, and a variant process recently suggested by Nick Wormald.  
We show that these two processes are asymptotically equally fast in constructing a semi-random graph $G$ that has property ${\mathcal P}$, for the following examples of ${\mathcal P}$:
\begin{itemize}
    \item ${\mathcal P}$ is the set of graphs containing a $d$-degenerate subgraph, where $d\ge 1$ is fixed;
    \item ${\mathcal P}$ is the set of $k$-connected graphs, where $k\ge 1$ is fixed.
\end{itemize}
In particular, our result of the $k$-connectedness above settles the open case $k=2$ of the original semi-random graph process.

We also prove that there exist properties ${\mathcal P}$ where the two semi-random graph processes do not construct a graph in ${\mathcal P}$ asymptotically equally fast. We further propose some conjectures on ${\mathcal P}$ for which the two processes perform differently.

% We investigate the property of the graph being $k$-connected. We settle the open case for $2$-connectedness by showing that the player has a strategy to construct a $2$-connected graph asymptotically almost surely in $(\ln 2 + \ln(\ln 2 + 1) + o(1))n$ rounds, which matches a known lower bound asymptotically. 

% Additionally, we consider a variant that was recently suggested by Wormald where the player chooses the first vertex and the second vertex is chosen u.a.r. We show that the bounds for $k$-connectedness for the traditional setting are also tight for this variant.

\end{abstract}

\newpage
\section{Introduction}
\label{section:introduction}
The semi-random graph process is a single player game initially suggested by Peleg Michaeli, and formally introduced by Ben-Eliezer, Hefetz, Kronenberg, Parczyk, Shikhelman and Stojakovi{\'{c}}~\cite{Ben-Eliezer2020Semi-randomProcess}. In this game, a graph is iteratively constructed from an empty graph on $n$ vertices, denoted by $[n]=\{1,2,\ldots, n\}$. Every round, one edge is added to the graph. The first end-vertex of the edge, $u$, is chosen uniformly at random (u.a.r.) from all the vertices in $[n]$. Given the choice of $u$, the other end-vertex $v$ is chosen strategically by the player (either deterministically, or by some random strategy). 

The semi-random graph process is part of a larger category of random processes where a player has limited power of choice among a set of random options. This category of combinatorial random processes traces its origins to the work of Azar, Broder, Karlin and Upfal~\cite{Azar1994BalancedAllocations} on placing $n$ balls into $n$ bins. They showed that if the player can choose from two u.a.r.\ selected bins rather than just one, there exists a strategy to decrease the expected number of balls in the fullest bin by an exponential factor. Similar load-balancing schemes have been investigated by Mitzenmacher~\cite{Mitzenmacher2001TheBalancing}. Another well-known example of such random processes is the so-called Achlioptas graph process, suggested by Dimitris Achlioptas during a Fields Institute workshop. Instead of adding a single edge picked u.a.r.\ every round as in the classical Erd\H{o}s-R\'enyi random graph process~\cite{erdosrenyi}, he suggested that every round the player is offered $k\ge 2$ such edges, and one of the $k$ edges can be chosen and added to the graph.  The Achlioptas graph process was first investigated by Bohman and Frieze~\cite{Bohman2001AvoidingComponent}, who showed that allowing the player to choose from $k\ge 2$ edges enables the player to delay the appearance of a giant component.

In the seminal paper on the semi-random graph process, Ben-Eliezer, Hefetz, Kronenberg, Parczyk, Shikhelman and Stojakovi{\'{c}}~\cite{Ben-Eliezer2020Semi-randomProcess} provided asymptotic upper and lower bounds on the number of rounds needed to achieve certain objectives a.a.s.\ (asymptotically almost surely, see Section~\ref{section:notation} for a precise definition), including having minimum degree at least $k\ge 1$, having clique number $k\ge 1$, and being $k$-connected. Additionally, they demonstrated how the semi-random graph process can be used to model other random graph models. Specifically, they established how to couple the semi-random process to the Erd\H{o}s-R\'enyi random graph model, the $k$-out model, and the min-degree process.

Further research by Behague, Marbach, Pra{\l}at and Rucinski~\cite{Behague2021SubgraphHypergraphs} gave tight asymptotic bounds for the minimum number of rounds needed to construct a graph that contains a subgraph isomorphic to a fixed graph $H$ based on the degeneracy of $H$. Moreover, they generalised the semi-random graph process to hypergraphs, and similarly showed tight bounds for constructing a fixed $s$-uniform hypergraph.

In terms of spanning subgraphs, Ben-Eliezer, Gishboliner, Hefetz and Krivelevich~\cite{Ben-Eliezer2020VeryProcess} showed that one can construct any fixed bounded-degree spanning subgraph a.a.s.\ in linear time. Moreover, MacRury, Pra{\l}at and the first author~\cite{Gao2022HamiltonProcess,Gao2021PerfectProcess,gao2022fully} obtained bounds on the minimum number of rounds needed to construct a graph with a perfect matching or a Hamilton cycle. The upper bound on the minimum number of rounds required for the construction of a Hamiltonian graph was further improved by Frieze and Sorkin~\cite{frieze2022hamilton}.

Recently, Gamarnik, Kang and Pra{\l}at~\cite{gamarnik2023cliques} have found bounds for the number of rounds needed to force the appearance of cliques and independent sets, and to ensure the graph has at least a given chromatic number.

Pra{\l}at and Singh~\cite{prałat2023power} have recently also considered the properties of minimum degree, the existence of a perfect matching and the existence of a Hamilton cycle in a generalisation of the semi-random graph process, where each round the player is presented with $k$ random vertices.

\subsection{Two semi-random processes}

Recently, Nick Wormald proposed (via personal contact) an alternative version of the semi-random graph process. Instead of the first vertex being chosen u.a.r.\ in each round and the second vertex being chosen according to some strategy, he proposed switching this order. That is, the first vertex in each round is chosen strategically by the player, whereas the second vertex is chosen u.a.r. We refer to this new model as the \emph{pre-positional semi-random graph process}, and the original model by Ben-Eliezer, Hefetz, Kronenberg, Parczyk, Shikhelman and Stojakovi{\'{c}}~\cite{Ben-Eliezer2020Semi-randomProcess} as the \emph{post-positional semi-random graph process}. By a simple coupling argument, it is easy to see that the post-positional process can construct a  graph in ${\mathcal P}$ at least as fast as the pre-positional process, for any graph property ${\mathcal P}$ (See Lemma~\ref{lem:relation_pre_post_positional} in Section~\ref{section:positional_model}).
The interesting question arising from comparing these two processes is, whether the post-positional process significantly outperforms the pre-positional process in constructing a member of ${\mathcal P}$. Perhaps a little surprisingly, for quite many properties ${\mathcal P}$, these two processes perform equally well. However, we also give an example of ${\mathcal P}$ where the post-positional process construct a graph in ${\mathcal P}$ significantly faster. 

\subsection{Main results}

Our first main result concerns the minimum number of rounds required to construct a $k$-connected graph. 

% Define
% \begin{equation}\label{alpha}
% \alpha_k=??? \quad \mbox{for $k\ge 2$, and}\quad \alpha_1=1.
% \end{equation}

\begin{theorem} \label{thm:k-connect}
   Let $k\ge 1$ be fixed. For every $\eps>0$, a.a.s.\ there exists a real number $\alpha_k$ such that the following hold:
   \begin{enumerate}
       \item[(a)] no strategy in a post-positional or pre-positional semi-random graph process can construct a $k$-connected graph in at most $(\alpha_k-\epsilon)n$ rounds;
       \item[(b)] there exist strategies in post-positional and pre-positional semi-random graph processes that construct a $k$-connected graph in at most $(\alpha_k+\epsilon)n$ rounds.
   \end{enumerate}
\end{theorem}
\begin{remark}
\begin{enumerate}
    \item[(a)] Theorem~\ref{thm:k-connect} was proved to be true for the post-positional semi-random graph process for $k\ge 3$ by Ben-Eliezer, Hefetz, Kronenberg, Parczyk, Shikhelman, and Stojakovi\'c~\cite{Ben-Eliezer2020Semi-randomProcess}. We settled the open case $k=2$.
    \item[(b)] The theorem confirms that the two variants of the semi-random graph process perform asymptotically equally well on constructing $k$-connected graphs.
    \item[(c)] The constant $\alpha_1$ is set to be 1. The real numbers $\alpha_k$ for $k \geq 2$ are derived from a system of differential equations and follow from applying Wormald's differential equation method (See~\cite{Wormald1999TheAlgorithms}). The first few values are
    \begin{align*}
        \alpha_2 &= \ln 2 + \ln (1+ \ln 2),\\
        \alpha_3 &= \ln (( \ln 2)^2 + 2(1+ \ln 2) (1+ \ln(1 + \ln 2))),
    \end{align*}
    as calculated in~\cite{Kang2006TheProcess}. 
\end{enumerate}

\end{remark}

Next, we show that the two processes perform equally well in constructing graphs with a given small subgraph.

A graph $H$ is said to be $d$-degenerate if each subgraph of $H$ contains a vertex of degree at most $d$. In their seminal paper, Ben-Eliezer, Hefetz, Kronenberg, Parczyk, Shikhelman, and Stojakovi\'c~\cite{Ben-Eliezer2020Semi-randomProcess} considered the number of rounds needed to construct a fixed size $d$-degenerate graph as a subgraph. They showed the following upper bound in the post-positional process.

\begin{theorem}[{\cite[Theorem~1.10]{Ben-Eliezer2020VeryProcess}}] \label{thm:degenerate-old}
    Let $H$ be a fixed $d$-degenerate graph, and let $f: \mathbb{N} \to \mathbb{R}$ be a function such that $\lim_{n \to \infty} f(n) = \infty$. Then there exists a strategy in the post-positional process such that the resulting graph $G$ contains a subgraph isomorphic to $H$ in a.a.s.\ $f(n) \cdot n^{(d-1)/d}$ rounds.
\end{theorem}

They conjectured that this bound is tight if $d \geq 2$, which was subsequently shown by Behague, Marbach, Pra{\l}at, and Ruci\'{n}ski~\cite{Behague2021SubgraphHypergraphs}. We show that the same bounds hold for the pre-positional process. 

\begin{theorem}
\label{thm:degeneracy}
    Let $H$ be a fixed $d$-degenerate graph, and let $f: \mathbb{N} \to \mathbb{R}$ be a function such that $\lim_{n \to \infty} f(n) = \infty$. Then a.a.s.\ the following hold:
  \begin{enumerate}
       \item[(a)] If $d\geq 2$, no strategy in a post-positional or pre-positional semi-random graph process can construct a graph containing a copy of $H$ in at most $ n^{(d-1)/d}/f(n)$ rounds;
       \item[(b)] there exist strategies in post-positional and pre-positional semi-random graph processes that construct a graph containing a copy of $H$ in at most $f(n)\cdot n^{(d-1)/d}$ rounds.
   \end{enumerate}
\end{theorem}

Theorem~\ref{thm:degeneracy} immediately gives the following corollary.
\begin{corollary}
    Let $H$ be a fixed graph containing a cycle, and $f: \mathbb{N} \to \mathbb{N}$ a function such that $\lim_{n\to\infty}f(n) = \infty$.
    Then a.a.s.\ the following hold:
  \begin{enumerate}
       \item[(a)] no strategy in a post-positional or pre-positional semi-random graph process can construct a graph containing an $H$-minor in at most $ n^{1/2}/f(n)$ rounds;
       \item[(b)] there exist strategies in post-positional and pre-positional semi-random graph processes that construct a graph containing an  $H$-minor in at most $f(n)\cdot n^{1/2}$ rounds.
   \end{enumerate}
\end{corollary}
\begin{proof} For (a), it suffices to show that a.a.s.\ $G_t$ is acyclic for $t\le n^{1/2}/f(n)$ in any post-positional process. Suppose $G_t$ has a cycle. Considering only the edges (each of which joins a square and a circle) that make up the cycle, there must exist a square which lands on a vertex that has already received either a square or a circle earlier. For every $t\le n^{1/2}/f(n)$, the probability that $u_t$ lands on a vertex that has already received a square or a circle (there are $2(t-1)$ of them) is $O(t/n)$ and hence, the probability that $G_t$ contains a cycle for some $t\le n^{1/2}/f(n)$ is bounded by $\sum_{t\le n^{1/2}/f(n)} O(t/n)=o(1)$.

    Part (b) follows from considering the $1$-subdivision of $H$, that is the graph obtained from $H$ by subdividing each edge in $E(H)$ exactly once, and noting that this subdivision is $2$-degenerate. The bound then directly follows from Theorem~\ref{thm:degeneracy}.
\end{proof}

Next, we investigate the performance of the two processes in constructing a graph containing a large bipartite subgraph.

\begin{theorem}\label{thm:bipartite}
Suppose $m=o(n^2)$. Let ${\mathcal P}$ be the set of graphs on $[n]$ that contain a bipartite subgraph with $m$ edges. Then, the minimum number of rounds required to construct a graph in ${\mathcal P}$ is a.a.s.\  $(1+o(1))m$ in both the pre-positional and the post-positional processes.     
\end{theorem}

While the proof for Theorem~\ref{thm:bipartite} is straightforward (See Section~\ref{section:dense_bipartite}), it is interesting to see if the theorem fails to hold when $m=\Theta(n^2)$. 
We think containing a bipartite subgraph with $\Omega(n^2)$ edges might be an increasing property (see its definition in Section~\ref{section:notation}) for which the post-positional process outperforms the pre-positional process. However, proving it seems not an easy task. We make the following conjecture.

\begin{conjecture}
    Suppose $m\ge cn^2$ for some fixed $c>0$. Let ${\mathcal P}$ be the set of graphs on $[n]$ that contain a bipartite subgraph with $m$ edges.  There exists $\delta_c,\eta_c>0$  such that a.a.s.\ there is a strategy in a post-positional process that construct a graph in ${\mathcal P}$ in less than $\eta_c n^2$ rounds, whereas no strategy in a pre-positional process can construct  a graph in ${\mathcal P}$ within $(\eta_c+\delta_c)n^2$ rounds. 
\end{conjecture}

Finally, we give an example of non-increasing ${\mathcal P}$ where the two processes perform very differently. 

\begin{theorem}\label{thm:induced}
    Let ${\mathcal P}$ be the set of multigraphs on $[n]$ that contains an induced simple $(n-1)$-cycle. Then, a.a.s.\ no pre-positional process can produce a multigraph in ${\mathcal P}$, whereas, a.a.s.\ a post-positional process can construct a multigraph in ${\mathcal P}$ in $O(n\log n)$ rounds. 
\end{theorem}

\section{Notation}
\label{section:notation}
For a graph $G$, we denote its vertex and edge sets by $V(G)$ and $E(G)$ respectively. We denote the degree of a vertex $v \in V(G)$ in graph $G$ by $\deg_G(v)$. We use $\delta(G)$ and $\Delta(G)$ to denote the minimum and maximum degrees of a graph respectively. For a set $S \subseteq V(G)$ of vertices, we use $G[S]$ for the subgraph induced by set $S$ in graph $G$. The open and closed neighbourhoods of a vertex $v\in V(G)$ in graph $G$ will be denoted by $N_G(v)$ and $N_G[v]$ respectively.

Both variants of the semi-random graph processes are a single-player game in which a multi-graph is iteratively constructed in a sequence of rounds. Because all the graph properties we consider are invariant under adding multi-edges and loops, we generally consider the underlying simple graph. Notably, we define the degree of a vertex in the multi-graph to be the number of distinct neighbours, not including itself. That is, $\deg_G(v) = |N_G(v) \setminus \{v\}|$ for each vertex $v \in V(G)$. Moreover, we will use the previously introduced notation for simple graphs for the graphs generated by the process as well.

In each round of the semi-random graph process (either variant), a single edge is added to the graph. We will denote the graph obtained after $\ell$ rounds by $G_{\ell}$. The initial graph, $G_0$, is an empty graph with vertex set $[n]$.  In the $t^{\text{th}}$ round, we construct graph $G_t$ from graph $G_{t-1}$ as follows. Let $u_t$ be a vertex picked u.a.r.\ from $[n]$. We say that vertex $u_t$ is hit in round $t$. We choose a vertex $v_t\in [n]$ according to some strategy, and add edge $u_tv_t$ to graph $G_{t-1}$ to obtain graph $G_t$. The strategy can be a function of $u_t$ in the post-positional variant, and must be independent of $u_t$ in the pre-positional variant. Note that if $u_t=v_t$ the new edge is a loop, and if $G_{t-1}$ already contained $u_tv_t$ the new edge is a multi-edge. Thus, $V(G_t) = V(G_{t-1})=[n]$, and $E(G_t) = E(G_{t-1}) \cup \{u_tv_t\}$. Additionally, we refer to $u_t$ as a square, and $v_t$ as a circle in round $t$, as introduced by Gao, MacRury and Pra{\l}at~\cite{Gao2021PerfectProcess}. Each edge in graph $G_t$ then connects a square and a circle in the round that it is added. 

We denote a graph $G$ having property $\mathcal{P}$ by $G \in \mathcal{P}$. We say that a graph property $\mathcal{P}$ is \emph{increasing} if for every $G \in \mathcal{P}$,  $H\in \mathcal{P}$ provided that $G\subseteq H$. Note that by this definition, if $G_t \in \mathcal{P}$ for some $t > 0$ and a monotone graph property $\mathcal{P}$, it follows that $G_{t'} \in \mathcal{P}$ for all $t' \geq t$ as well. Except for the example in Theorem~\ref{thm:induced}, all properties investigated in this paper are increasing properties.

%The aim of the game is often to a.a.s.\ force graph $G_t$ to have specified monotone graph properties in as few rounds as possible. We refer to this monotone graph property as the \emph{objective property}. We are specifically interested in the asymptotic value of the minimum number of rounds required to ensure $G_t$ has the objective property when taking $n$ to infinity. Let the \emph{history} at time $t$, denoted by $\mathcal{H}_t$, be a record of all events in the process up to round $t$ (that is, $u_1, v_1, u_2, v_2, \ldots, u_{t-1}, v_{t-1}$). A \emph{strategy}~$\mathcal{S}$ is then a sequence of functions $f_1, f_2, \ldots$ such that $f_t$ takes history $\mathcal{H}_t$ and vertex $u_t$ as input, and gives a probability distribution over $[n]$. Vertex $v_t$ is then chosen according to this probability distribution in round $t$. 

%***********

%We note that upper bounds on $\tau_{\mathcal{P}}$ can be obtained by considering explicit strategies. In this paper, we will design and analyse various strategies to provide bounds for several properties $\mathcal{P}$. In this paper, we predominantly considered objective properties that are monotonically increasing. 

If ${\mathcal P}$ is increasing, it is sufficient to construct a graph $G_t$ which has a subgraph $G'$ in $\mathcal{P}$. In some rounds, given $G_{t-1}$ (and vertex $u_t$ in the post-positional process), we may choose vertex $v_t$ arbitrarily and not use the edge $u_tv_t$ for the construction of $G'$. We will consider such a round a \emph{failure round}. Allowing failure rounds in some cases leads to algorithms that are easier to analyse. In Section~\ref{section:induced_cycles} where a non-increasing property is studied, we cannot simply ignore ``undesirable'' edges to make use of failure rounds.

We say an event $A = A_n$ occurs asymptotically almost surely (a.a.s.) in $G_t$ if $\mathbb{P}(A_n) \to 1$ as $n \to \infty$. 
Unless specified otherwise, all asymptotic notation relates to $n$, i.e.\ $o(1)$ implies a function that tends to $0$ as $n \to \infty$.

\section{Pre- and post-positional processes}
\label{section:positional_model}

In this section we prove that the post-positional process can construct a  graph in ${\mathcal P}$ at least as fast as the pre-positional process, for any graph property ${\mathcal P}$.
%For the pre-positional process we use notation analogous to the post-positional process. A strategy $\mathcal{S}$ for the pre-positional process is a sequence of functions $f_1, f_2, \ldots$, where each function $f_t$ depends only on  $G_{t-1}$, and gives a probability distribution over $[n]$. For a property $\mathcal{P}$, a strategy $\mathcal{S}$, and probability $q$, we then define $\tau_{\mathcal{P}}'(\mathcal{S}, q, n)$, $\tau_{\mathcal{P}}'(q, n)$, and $\tau_{\mathcal{P}}'$ for the pre-positional process analogous to $\tau_{\mathcal{P}}(\mathcal{S}, q, n)$, $\tau_{\mathcal{P}}(q, n)$, and $\tau_{\mathcal{P}}$ in the post-positional process respectively.

\begin{lemma}
\label{lem:relation_pre_post_positional}
%    For any graph property $\mathcal{P}$, $\tau_{\mathcal{P}} \leq \tau_{\mathcal{P}}'$.
Let $n\ge 2$ and ${\mathcal P}\subseteq 2^{\binom{[n]}{2}}$. For every $t\ge 0$, the probability that there exists a pre-positional strategy to construct $G_s\in {\mathcal P}$ for some $s\le t$ is at most the probability that there exists a post-positional strategy to construct $G_s\in {\mathcal P}$ for some $s\le t$. 
\end{lemma} 
% \jc{No. We do not need a.a.s.\ here. You can ask me why in the future in person. I deleted all comments. }
% \hc{Okay, sounds good.}
%\begin{proof}
 %   Let $\mathcal{S}$ be an optimal strategy for property $\mathcal{P}$ in the pre-positional process for fixed $n$ and probability $q$ such that $\tau_{\mathcal{P}}'(\mathcal{S}, q, n) = \tau_{\mathcal{P}}'(q, n)$. We note that $\mathcal{S}$ is also a valid strategy for the post-positional process, where the potential dependence on vertex $u_t$ in round $t$ is not used. Hence, $\tau_{\mathcal{P}}(\mathcal{S}, q, n) \geq \tau_{\mathcal{P}}(q, n)$.
    
  %  We moreover observe that for strategy $\mathcal{S}$ not depending on vertex $u_t$ in round $t$, the probability distribution over all possible edges in round $t$ is equal between the two models. As a result, $\tau_{\mathcal{P}}(\mathcal{S}, q, n) = \tau_{\mathcal{P}}'(\mathcal{S}, q, n)$. Thus,
   % \[\tau_{\mathcal{P}}(q, n) \leq \tau_{\mathcal{P}}(\mathcal{S}, q, n) = \tau_{\mathcal{P}}'(\mathcal{S}, q, n) = \tau_{\mathcal{P}}'(q, n).\]
    
    %Then, as $\tau_{\mathcal{P}}(q, n) \leq \tau_{\mathcal{P}}'(q, n)$ for all $0 < q < 1$ and $n \geq 1$, we find
    %\[\tau_{\mathcal{P}} = \lim_{q\, \toup\, 1}\limsup_{n \to \infty} \frac{\tau_{\mathcal{P}}(q, n)}{n} \leq \lim_{q\,\toup\, 1}\limsup_{n \to \infty} \frac{\tau_{\mathcal{P}}'(q, n)}{n} = \tau_{\mathcal{P}}',\]
    %as desired.
%\end{proof}

%\jc{The proof above is good. Please keep it for your thesis. Below is a version that I may prefer for the paper though.}

\begin{proof}
    We can couple the two processes so that no matter which strategy the pre-positional process uses, the post-positional process can copy the moves and can stop the process at the same time as the pre-positional one. Let $(u_i)_{i\ge 0}$ be a sequence of i.i.d.\ copies of $u$ chosen u.a.r.\ from $[n]$. Present $u_i$ to be the $i$-th square for both processes. For each $t\ge 1$, let $v_t$ be the choice of the $t$-th circle by the pre-positional process. Note that the choice of $v_t$ depends only on $\{u_i,v_i, 1\le i\le t-1\}$. The post-positional process simply copies the choices of $v_t$ for every $t$, which are valid moves given its definition. Thus, the two processes always terminate at the same time. 
\end{proof}

Thanks to Lemma~\ref{lem:relation_pre_post_positional}, it suffices to prove Theorem~\ref{thm:k-connect}(a) and~\ref{thm:degeneracy}(a) for post-positional processes, and prove Theorem~\ref{thm:k-connect}(b), ~\ref{thm:degeneracy}(b) and Theorem~\ref{thm:bipartite} for pre-positional processes. 

We prove Theorem~\ref{thm:k-connect} in Section~\ref{section:k-connectivity}, Theorem~\ref{thm:degeneracy} in Section~\ref{section:degeneracy}, and Theorem~\ref{thm:bipartite} in Section~\ref{section:dense_bipartite}, and Theorem~\ref{thm:induced} in Section~\ref{section:induced_cycles}.

\section{$k$-Connectivity: proof of Theorem~\ref{thm:k-connect}}
\label{section:k-connectivity}
A connected graph $G$ is said to be $k$-connected if it remains connected when removing fewer than $k$ vertices. In their seminal paper, Ben-Eliezer, Hefetz, Kronenberg, Parczyk, Shikhelman and Stojakovi{\'{c}}~\cite{Ben-Eliezer2020Semi-randomProcess} provide tight asymptotic bounds for the minimum number of rounds needed in the post-positional process to produce a $k$-connected graph for all $k \geq 3$. Their lower bounds follow directly by coupling with a well-known random graph process called the $k$-min process. By Lemma~\ref{lem:relation_pre_post_positional}, these lower bounds are valid for the pre-positional process as well. As a warming up, we will go through their argument and show how it also works directly in the pre-positional setting.

\subsection{Min-degree process: proof of Theorem~\ref{thm:k-connect}(a)}
\label{subsection:min-degree_process}
The min-degree process is a variant on the classical random graph process and was introduced and first studied by Wormald~\cite{Wormald1999TheAlgorithms}. In the min-degree process, $G_0$ is an edgeless graph on $[n]$. Given $G_t$, choose a vertex $u$ of minimum degree in $G_t$ u.a.r., and subsequently choose a vertex $v$ not adjacent to vertex $u$ in graph $G_t$ u.a.r. Graph $G_{t+1}$ is then constructed by adding edge $uv$ to graph $G_t$. Recall $\alpha_k$ in Theorem~\ref{thm:k-connect}. Wormald used his differential equation method to prove that the minimum $t$ where $G_t$ has minimum degree at least $k$ is a.a.s.\ $\alpha_k n$, for each $k \geq 2$. We denote the graph property of having minimum degree $k$ by $\mathcal{D}_k$.

Ben-Eliezer, Hefetz, Kronenberg, Parczyk, Shikhelman and Stojakovi{\'{c}}~\cite{Ben-Eliezer2020Semi-randomProcess} have since studied adapted versions of the min-degree process as modelled by the semi-random graph process. By choosing $v_t$ u.a.r.\ from all vertices of minimum degree not adjacent to $u_t$ in graph $G_t$, the resulting semi-random graph process is contiguous to the min-degree process. That is, asymptotically the two processes are equivalent. We refer to this strategy as $\mathcal{S}_{\text{min}}$. Ben-Eliezer, Hefetz, Kronenberg, Parczyk, Shikhelman and Stojakovi{\'{c}} additionally considered strategies without the restrictions on $v_t$ and $u_t$ to be non-adjacent in $G_t$ and $u_t$ and $v_t$ to be distinct. They showed that each of these strategies are optimal in ensuring graph $G_t$ having minimum degree $k$ in as few rounds as possible when taking $n$ to infinity, and each asymptotically require $\alpha_k n$ rounds ($k\ge 2$). Each of these strategies thus obtains a graph in $\mathcal{D}_k$ in asymptotically the same number of rounds as the min-degree process. 

We first provide a formal definition of strategy $\mathcal{S}_{\text{min}}$. For each round $t$, distribution function $f_t$ is defined as follows. Let $Y_{t, \, \text{min}} = \{v \in [n] \, | \, \deg_{G_{t-1}}(v) = \delta(G_{t-1})\}$. Then, given $u_t$ chosen u.a.r.\ from $[n]$, if $Y_{t, \, \text{min}} \setminus N_{G_{t-1}}[u_t] = \emptyset$, the round is considered a failure round. Otherwise, vertex $v_t$ is chosen u.a.r.\ from $Y_{t, \, \text{min}} \setminus N_{G_{t-1}}[u_t]$. By this formulation, strategy $\mathcal{S}_{\text{min}}$ does not create loops nor multi-edges. 
Let $G_{\text{min}}(n, m)$ be the graph on $n$ vertices with $m$ edges generated by the min-degree process. To show that strategy $\mathcal{S}_{\text{min}}$ can be used to model the min-degree process for $m = o(n^2)$ with a.a.s.\ $o(m)$ failure rounds, Ben-Eliezer, Hefetz, Kronenberg, Parczyk, Shikhelman and Stojakovi{\'{c}}~\cite{Ben-Eliezer2020Semi-randomProcess} look at an auxiliary strategy where $v_t$ is chosen u.a.r.\ from all minimum degree vertices. This strategy thus does not take the neighbourhood of $u_t$ into account. They then show that the number of multi-edges and loops is asymptotically bounded by $o(m)$, which directly bounds the failure rounds of strategy $\mathcal{S}_{\text{min}}$ as well. We note that this auxiliary strategy is also valid in the pre-positional process, where the first vertex is chosen u.a.r.\ from the vertices of minimum degree. Hence, the pre-positional process can also model the min-degree process with a.a.s.\ $o(m)$ failure rounds. Since having minimum degree $k$ is a prerequisite for being $k$-connected for $n > k$, this immediately implies Theorem~\ref{thm:k-connect}(a) for $k\ge 2$. The case $k=1$ is trivial, as a connected graph has at least $n-1$ edges and thus no strategy can build a connected graph in at most $(1-\eps)n$ rounds. \qed

\subsection{Proof of Theorem~\ref{thm:k-connect}(b)}

We consider the set of $k$-connected graphs on $[n]$, which is an increasing property. It is convenient to define some notation to assist the proof of Theorem~\ref{thm:k-connect}(b).

For an increasing property $\mathcal{P}$, a strategy $\mathcal{S}$ ($\mathcal{S}$ may be a pre-positional or a post-positional strategy), and a real number $0 < q < 1$, let $\tau_{\mathcal{P}}(\mathcal{S},q, n)$ be the minimum value $t \geq 0$ such that $\mathbb{P}\left[G_t \in \mathcal{P}\right] \geq q$; recalling that $n$ is the number of vertices in $G_t$. If no such value $t$ exists, we say that $\tau_{\mathcal{P}}(\mathcal{S}, q, n) = \infty$. Let $\tau_{\mathcal{P}}(q, n)$ denote the minimum value of $\tau_{\mathcal{P}}(\mathcal{S}, q, n)$ over all possible strategies $\mathcal{S}$. We are interested in the asymptotic value of $\tau_{\mathcal{P}}(q, n)$ when probability $q$ approaches $1$. Therefore, we define
\[\tau_{\mathcal{P}} := \lim_{q\, \toup\, 1}\limsup_{n \to \infty} \frac{\tau_{\mathcal{P}}(q, n)}{n},\] where the limit exists since ${\mathcal P}$ is increasing. This definition is useful for studying linear-time strategies (strategies that a.a.s.\ builds a graph in ${\mathcal P}$ in $\Theta(n)$ rounds), which is the case for 
\[
{\mathcal P}={\mathcal C}_k:=\{G\subseteq \binom{[n]}{2}:\ G\ \text{is $k$-connected}\}. 
\]

To show Theorem~\ref{thm:k-connect}(b), it suffices to prove that in the pre-positional process,
\[
\tau_{{\mathcal C}_k}\le \alpha_k\quad \text{for every fixed $k\ge 1$}. \]

%In this paper, we will only consider the min-degree process modelled by the semi-random graph process which avoids loops and multi-edges, as given by strategy $\mathcal{S}_{\text{min}}$. 
Let \emph{$k$-min process} be the process of applying strategy $\mathcal{S}_{\text{min}}$ until obtaining a graph with minimum degree at least $k$. Ben-Eliezer, Hefetz, Kronenberg, Parczyk, Shikhelman and Stojakovi{\'{c}}~\cite{Ben-Eliezer2020Semi-randomProcess} proved  Theorem~\ref{thm:k-connect} for the case $k \ge 3$ in the post-positional process. Their proof is based on a slightly modified variant of the $k$-min process tailored for multigraphs and builds on a proof by Kang, Koh, Ree and {\L}uczak~\cite{Kang2006TheProcess}. The strategy $\mathcal{S}_{\text{min}}^*$ underlying their modified process is identical to the strategy for the $k$-min process as long as the graph is simple, and simplifies the analysis in the semi-random graph process. The proof shows that the graph resulting from the modified $k$-min process is a.a.s.\ $k$-connected for all $k \geq 3$. This proof cannot be directly extended to $k < 3$, as Kang, Koh, Ree and {\L}uczak~\cite{Kang2006TheProcess} showed that the graph resulting from the $k$-min process is only a.a.s.\ connected for $k \geq 3$.

Strategy $\mathcal{S}_{\text{min}}^*$ chooses vertex $v_t$ u.a.r.\ from all vertices of $V(G_{t-1}) \setminus \{u_t\}$ that have the smallest number of distinct neighbours. This strategy can be modelled in the pre-positional process by the following strategy: we choose $u_t$ u.a.r.\ from all vertices that have the smallest number of distinct neighbours, and consider the round a failure if $u_t = v_t$. The probability of any given round being a failure round is thus $1/n$. Hence, the number of additional rounds needed to cover the additional failure rounds is a.a.s.\ $o(n)$.  Hence, it immediately gives the following lemma.
\begin{lemma}
\label{lem:pre-positional_k-connectedness} In both the pre-positional and the post-positional process,
    $\tau_{\mathcal{C}_k}=\alpha_k$ for all fixed $k \geq 3$.
\end{lemma}

Moreover, note that the case $k = 1$  is trivial in the post-positional process. Namely, we observe that one can build a forest containing $m \leq n -1$ edges in exactly $m$ rounds. In each round, we simply choose $v_t$ that lies in a different component of $G_{t-1}$ from  $u_t$. Hence, we can build a spanning tree in $n-1$ rounds, which is obviously optimal. The following lemma shows that the pre-positional process requires asymptotically the same number of rounds to construct a connected graph.

\begin{lemma}
\label{lem:pre-positional_1-connectedness}
    $\tau_{\mathcal{C}_1} = 1$ in the pre-positional process.
\end{lemma}
\begin{proof}
    It is obvious that $\tau_{\mathcal{C}_1} \geq 1$, since a connected graph on $[n]$ has at least $n-1$ edges. 

    Recall that $u_t$ is the vertex uniformly chosen from $[n]$, and $v_t$ is the vertex strategically chosen by the player. 
    For the upper bound, we consider a strategy $\mathcal{S}$ which chooses $v_t$ u.a.r.\ from the smallest component (if there is a tie, pick an arbitrary smallest component). If $u_t$ lands in a different component, we add edge $u_tv_t$. Otherwise we consider the round a failure round. Each successfully added edge then decreases the number of components in the graph by one. We analyse the process with this strategy in a number of phases.

    Let phase $i$ be defined as the rounds in which the number of components in the graph decreases from $\frac{n}{2^{i-1}}$ to $\frac{n}{2^i}$. Thus, there are $\log_2 n$ such phases. We note that phase $i$ consists of $\frac{n}{2^{i-1}} - \frac{n}{2^i} = \frac{n}{2^i}$ non-failure rounds, and a number of failure rounds. Let $T_i$ be the total number of rounds in phase $i$, and let $f_i$ be the number of failure rounds in phase $i$. Thus, $T_i = \frac{n}{2^{i}} + f_i$.
    
    Next, we observe that the smallest component  in any round in phase $i$ contains at most $2^i$ vertices. The probability that a round is a failure round in phase $i$ is thus at most $2^i /n$. Couple the process with the following experiment; consider a sequence of i.i.d.\ Bernoulli random variables with success probability $1 - 2^i/n$. We terminate the sequence once we have observed $n/2^i$ successes. Let $\mathcal{T}_i$ denote the random variable corresponding to the number of Bernoulli random variables in the sequence before it terminates. We observe that $T_i$ is stochastically dominated by $\mathcal{T}_i$. By the negative binomial distribution, it follows that $\mathbb{E}[\mathcal{T}_i] = (n/2^i)/(1-2^i/n)$. Hence, $\mathbb{E}[T_i] \le \mathbb{E}[\mathcal{T}_i] \leq (n/2^i)/(1-2^i/n)$. Then, as $T_i = \frac{n}{2^{i}} + f_i$, we find that for all $i \leq \log_2(n) - 1$:
    \[\mathbb{E}[f_i] \leq \frac{\frac{n}{2^i}}{1-\frac{2^i}{n}} - \frac{n}{2^i} = 1+O\left(\frac{2^i}{n}\right)=O(1). %= \frac{\frac{n}{2_i} \left( 1 - (1- \frac{2^i}{n})\right)}{1-\frac{2^i}{n}} = \frac{1}{1 - \frac{2^i}{n}}.
    \]
%    Note that this bound is only useful if $i < \log_2(n)$. We then observe that for , it follows that 
 %   \[\mathbb{E}[f_i] \leq \frac{1}{1 - \frac{2^i}{n}} \leq \frac{1}{1 - \frac{2^{\log_2(n)-1)}}{n}} = 2.\]
    For the last phase (i.e.\ $\log_2(n)-1<i\le \log_2 n $) only a single successful round is needed, and as the failure probability is at most $1/2$, it follows that for all $i$ it holds that $\mathbb{E}[f_i] = O(1)$. Therefore, $\mathbb{E}[\sum_{i\le \log_2 n}f_i]=O(\log_2 n)$, and thus By Markov's inequality, a.a.s.\ $\sum_{i\le \log_2 n}f_i = O(\log^2 n)$.
    Hence, total number of rounds needed to ensure the graph is connected is a.a.s.\ at most 
    \[\sum_{i > 0}T_i = n-1 + \sum_{i = 1}^{\log_2 n} f_i = (1 + o(1))n. \]
    Therefore, $\tau_{\mathcal{C}_1} \leq 1$ in the pre-positional process, as desired.
\end{proof}

Thus, asymptotically, the number of required rounds to ensure connectivity is equal between the pre- and post-positional processes.

In this section we prove tight asymptotic bounds for the final open case in both the pre- and post-positional processes, $k = 2$. The best bound previously known for the post-positional process, as observed by Ben-Eliezer, Hefetz, Kronenberg, Parczyk, Shikhelman and Stojakovi{\'{c}}~\cite{Ben-Eliezer2020Semi-randomProcess}, is the tight upper bound for $k = 3$. That is, $\tau_{\mathcal{C}_2} \leq \tau_{\mathcal{C}_3}$. They also gave a lower bound, based on the $2$-min process. The $2$-min process aims to ensure that each vertex has degree at least two as fast as possible, a prerequisite for $2$-connectedness. Using a known result by Wormald~\cite{Wormald1995DifferentialGraphs, Wormald1999TheAlgorithms} on the min-degree process, they showed that the $2$-min process a.a.s.\ takes $(\ln 2 + \ln(\ln 2+1)) + o(1))n$ rounds to complete. Hence, $\tau_{\mathcal{C}_2} \geq \ln 2 + \ln(\ln 2+1)$ in the post-positional process. Note that as the $2$-min process can be modelled by the pre-positional process as well, it similarly holds that $\tau_{\mathcal{C}_2} \geq \ln 2 + \ln(\ln 2+1)$ in the pre-positional process.

In this section we show a novel upper bound for the pre-positional process, which asymptotically matches the known lower bound. Note that by Lemma~\ref{lem:relation_pre_post_positional}, this directly gives an asymptotically tight upper bound for the post-positional process as well.

\begin{lemma}
    \label{thm:2-connectedness}
    $\tau_{\mathcal{C}_2} = \ln 2 + \ln(\ln 2+1)$ in both the pre- and post-positional processes.
\end{lemma}

That is, the minimum number of rounds required for a semi-random process to build a $2$-connected graph on $n$ vertices is asymptotic to $(\ln 2 + \ln(\ln 2+1))n$ in both processes.

As a result of Lemma~\ref{thm:2-connectedness}, and the previous analysis of existing proofs for bounds on $\tau_{\mathcal{C}_k}$ for $k \geq 1$, it follows that the property of $k$-connectedness requires asymptotically the same number of rounds in the pre- and post-positional processes. 

\subsubsection{Overview}
For the upper bound, our approach differs significantly from the strategy used by Ben-Eliezer, Hefetz, Kronenberg, Parczyk, Shikhelman and Stojakovi{\'{c}}~\cite{Ben-Eliezer2020Semi-randomProcess} to prove the tight upper bounds for $k$-connectedness for $k \geq 3$. Namely, while their approach is predominantly probabilistic, we use a more structural approach. Our strategy is based on analysing the structure of the maximal $2$-connected components of the graph resulting from the $2$-min process.

In the first phase, we use the $2$-min process to obtain a graph in which each vertex has degree at least $2$. We show that a.a.s.\ most of the vertices in this graph will be contained in relatively large $2$-connected subgraphs. This moreover allows us to conclude that the graph contains $o(n)$ maximal $2$-connected subgraphs. 

In the second phase, the aim is to ensure that the graph becomes connected. We bound the number of components by the number of maximal $2$-connected subgraphs, recalling that the graph has $o(n)$ such subgraphs after the first phase. As such, by adding edges between components, we can quickly ensure the graph becomes connected.

In the third phase, we then want to make the graph $2$-connected. We achieve this by considering a tree structure on the maximal $2$-connected subgraphs, and showing that by balancing this tree, we can efficiently eliminate cut-vertices.

We show that the second and third phases both take $o(n)$ steps a.a.s. Therefore, the first phase, consisting of the $2$-min process, dominates the total number of rounds in the process of building a $2$-connected graph on $[n]$. 

In Section~\ref{subsection:supporting_structural_results}, we first introduce the purely structural definitions and results we will use. Section~\ref{subsection:proof_upper_bound_k-connectedness} then builds upon these structural results to analyse the random process given by our strategy.

\subsubsection{Supporting structural results}
\label{subsection:supporting_structural_results}
In this section we restate the conventional definitions of blocks and block graphs (see for instance~\cite{Konig1936TheorieStreckenkomplexe}).

\begin{definition}[Block]
	\label{def:block}
	Let $B \subseteq V(G)$ be a maximal set of vertices such that for any two vertices $x, y \in B$ with $xy \not \in E(G)$, in order to separate vertex $x$ from vertex $y$, it is necessary to remove at least $2$ vertices from $G$. Then $B$ is called a block.
\end{definition}

Note that by this definition, each block in a graph either induces a maximal $2$-connected subgraph, an edge, or an isolated vertex. Moreover, when considering connected graphs on at least $2$ vertices, each block thus induces a maximal $2$-connected subgraph or an edge. Based on this definition, we can then decompose a graph $G$ into such blocks.

\begin{definition}[Block decomposition]
	\label{def:block_decomposition}
	Let $\mathcal{B}(G) \subseteq \mathcal{P}(V(G))$ denote the set of all blocks of graph $G$. Then $\mathcal{B}(G)$ is called the block decomposition of graph $G$.
\end{definition}

We observe that by the definition of blocks, for each edge $uv \in E(G)$ in a graph $G$ there exists a unique block $B \in \mathcal{B}(G)$ such that $u,v \in B$. Moreover, by the maximality of the blocks, the block decomposition $\mathcal{B}(G)$ is unique. Note that $\mathcal{B}(G)$ is generally not a partition of $V(G)$. However, each pair of blocks shares at most one vertex, as given in the following proposition.

\begin{proposition}[{{}K{\H{o}}nig, \cite[Theorem~XIV.7]{Konig1936TheorieStreckenkomplexe}}]
	\label{prop:max_intersect_block}
	Let $G$ be a graph. Then, for each pair of blocks $B_1, B_2 \in \mathcal{B}(G)$, it holds that $|B_1 \cap B_2| \leq 1$.
\end{proposition}
\begin{comment}
\begin{proof}
    Suppose not. Then let $B_1, B_2 \in \mathcal{B}$ be blocks such that $|B_1 \cap B_2| \geq 2$. By the maximality of blocks $B_1$ and $B_2$, neither $B_1 \subset B_2$ nor $B_2 \subset B_1$. Hence, $|B_1|, |B_2| \geq 3$, and thus both blocks $B_1$ and $B_2$ induce $2$-connected subgraphs. Moreover, by the maximality of blocks $B_1$ and $B_2$, vertex set $B_1 \cup B_2$ does not induce a $2$-connected graph. Let $u_1, u_2 \in B_1 \cap B_2$. Additionally, let $v_1, v_2 \in B_1 \cup B_2$ be vertices such that deleting one vertex from graph $G$ separates vertices $v_1$ and $v_2$. 
    
    As blocks $B_1$ and $B_2$ induce $2$-connected subgraphs, $|B_1\cap \{v_1, v_2\}| = |B_2 \cap \{v_1, v_2\}| = 1$. Then, without loss of generality, assume that $v_1 \in B_1 \setminus B_2$, and $v_2 \in B_2 \setminus B_1$. Additionally, without loss of generality assume that deleting vertex $w \in B_1$ separates vertices $v_1$ and $v_2$. We note that as block $B_1$ is $2$-connected, $B_1 \setminus \{w\}$ induces a connected subgraph in graph $G$. Moreover, $B_1 \setminus \{w\}$ contains at least one of the vertices $u_1, u_2$. Without loss of generality, assume $u_1 \in B_1 \setminus \{w\}$. But then, as $B_2 \setminus \{w\}$ is connected, and $v_2 \neq w$, there exist $v_1-u_1$ and $u_1-v_2$ paths, and thus a $v_1-v_2$ paths in graph $G - w$, contradicting the definition of vertices $v_1$, $v_2$ and $w$. we thus conclude that the proposition holds.
\end{proof}
\end{comment}

\begin{definition}[Block graph]
	\label{def:block_graph} 
	Let $G$ be a graph. Then let $G_{\mathcal{B}}$ be the graph defined by $V(G_{\mathcal{B}}) = \mathcal{B}(G)$ and $E(G_{\mathcal{B}}) = \{B_1B_2 \, | \, B_1 \cap B_2 \neq \emptyset\}$.  Then graph $G_{\mathcal{B}}$ is called the block graph of graph $G$.
\end{definition}

For a graph $G$ to be $2$-connected, it must hold that $\mathcal{B}(G) = \{V(G)\}$. We aim to use the blocks and their relative structure in a graph to identify moves in a semi-random process which join multiple blocks together into a single larger block. If a semi-random edge $u_tv_t$ joins two blocks then we call the addition of such an edge an \emph{augmentation}. A natural augmentation to consider is to join two blocks $B_i$ and $B_j$ where there is a path between $B_i$ and $B_j$ in $G_{{\cal B}}$. If $u_t$ and $v_t$ are not themselves cut-vertices, this augmentation will immediately join all blocks along the path into a single block. To that purpose, we want to consider a tree structure on the blocks.

The traditional such structure, called the block-cut tree of a graph, was originally introduced independently by Gallai~\cite{Gallai1964ElementareGraphen}, and Harary and Prins~\cite{Harary1966TheGraph}.

\begin{definition}[Block-cut tree]
    \label{def:block-cut_tree}
    Let $G$ be a connected graph, and let $S$ be the set of cut vertices of graph $G$. Then, the graph $T$, given by $V(T) = \mathcal{B}(G) \cup S$ and $E(T) = \{vB \, | \, v \in S, B \in \mathcal{B}(G), v \in B\}$, is a tree and called the block-cut tree of graph $G$.
\end{definition}

We consider a structure similar to the block-cut tree, based on the block graph. Instead of including the cut-vertices in the tree, we take a spanning tree on the block graph. This ensures that we only have to work with blocks, while still providing the desired tree structure. To that aim, we introduce the following definition.

\begin{definition}[Reduced block tree]
    \label{def:reduced_block_tree}
    Let $G_{\mathcal{B}}$ be the block graph of a connected graph $G$. Then, a spanning tree $T_{\mathcal{B}}$ of graph $G_{\mathcal{B}}$ is called a reduced block tree of graph $G$.
\end{definition}

A reduced block tree can equivalently be constructed recursively. Let $v \in V(G)$ be a cut-vertex in a connected graph $G$, and let $G_1$ and $G_2$ be the induced subgraphs of graph $G$ such that $V(G_1) \cup V(G_2) = V(G)$, $E(G_1) \cup E(G_2) = E(G)$, and $V(G_1) \cap V(G_2) = \{v\}$. We note that as vertex $v$ is a cut-vertex, each block $B \in \mathcal{B}(G)$ is contained in either graph $G_1$ or graph $G_2$. Therefore, $\mathcal{B}(G_1) \cup \mathcal{B}(G_2) = \mathcal{B}(G)$.  Let $T_{\mathcal{B}_1}$ and $T_{\mathcal{B}_2}$ be reduced block trees for graphs $G_1$ and $G_2$ respectively. Then, we can construct a reduced block tree for graph $G$ with block decomposition $\mathcal{B}(G)$ by joining trees $T_{\mathcal{B}_1}$ and $T_{\mathcal{B}_2}$ with a single edge from a vertex in $T_{\mathcal{B}_1}$ representing a block containing vertex $v$ to a vertex in $T_{\mathcal{B}_2}$ also representing a block containing vertex $v$. We observe that by Definition~\ref{def:reduced_block_tree}, the reduced block tree of a graph is generally not unique. This occurs when a vertex is contained in at least three blocks, and the block graph thus contains a clique of size at least $3$.

\begin{proposition}
    \label{prop:reduced_block_tree_subtrees}
    Let $T_{\mathcal{B}}$ be a reduced block tree of a connected graph $G$. For $v \in V(G)$, the set $\{B \in V(T_{\mathcal{B}}) \, | \, v \in B\}$ induces a (connected) subtree in $T_{\mathcal{B}}$.
\end{proposition}
\begin{proof}
    Suppose not. Let $S \subseteq V(T_{\mathcal{B}})$ be the set of all blocks $B \in V(T_{\mathcal{B}})$ such that $v \in B$. Then the set $S$ induces a disconnected subgraph in tree $T_{\mathcal{B}}$. Let $C_1$ and $C_2$ be two components of this induced subgraph $T_{\mathcal{B}}[S]$. Moreover, let $P$ be a shortest path between sets $V(C_1)$ and $V(C_2)$ in $T_{\mathcal{B}}$, and let blocks $B_1, B_2 \in S$ be the endpoints of this path $P$ such that $B_1 \in V(C_1)$ and $B_2 \in V(C_2)$. We note that $P$ has length at least 2. Then, as $P$ is a shortest such path, none of the internal vertices of $P$ are contained in $S$. Hence, the corresponding blocks do not contain vertex $v$. Let $G_P$ be the subgraph of $T_{\mathcal{B}}$ induced by the internal vertices of path $P$. Additionally, let $S_P \subseteq V(G)$ be the set of all vertices of graph $G$ contained in at least one of the blocks in $G_P$.
    
    We observe that by the definition of path $P$, subgraph $G_P$ contains blocks adjacent to blocks $B_1$ and $B_2$, respectively, in the tree $T_{\mathcal{B}}$. Therefore, $B_1 \cap S_P, B_2 \cap S_P \neq \emptyset$. Moreover, by Proposition~\ref{prop:max_intersect_block} we find that $B_1 \cap B_2 = \{v\}$. Therefore, as $v \not \in S_P$, there exist vertices $v_1 \in B_1 \cap S_P$ and $v_2 \in B_2 \cap S_P$. Then, because blocks $B_1$ and $B_2$ are by definition connected, there exists a $v-v_1$ path $P_1$ in block $B_1$ and a $v-v_2$ path $P_2$ in block $B_2$. Similarly, the set $S_P$ induces a connected subgraph in $G$, and thus contains a $v_1-v_2$ path $P'$. We note that the union of the paths $P_1$, $P_2$ and $P'$ gives a subgraph of $G$ containing a cycle $C$ containing vertex $v$. We note that the cycle $C$ is $2$-connected and hence is contained in a block $B_C$. Moreover, as this cycle contains at least $2$ vertices of block $B_1$, by Proposition~\ref{prop:max_intersect_block}, we find that $B_1 = B_C$. Analogously, it follows that $B_2 = B_C$. However, this implies that $B_1 = B_2$, contradicting these blocks being in different components $C_1$ and $C_2$. By this contradiction, we conclude that the proposition holds.
\end{proof}

\begin{proposition}
    \label{prop:reduced_block_tree_k2}
    Let $T_{\mathcal{B}}$ be a reduced block tree of a connected graph $G$ with $\delta(G)\geq 2$. Let $B \in \mathcal{B}$ be a block such that $B = \{u, v\}$. Then there exist distinct blocks $B_u, B_v \in \mathcal{B}$ adjacent to $B$ in $T_{\mathcal{B}}$ such that $u \in B_u$ and $v \in B_v$.
\end{proposition}
\begin{proof}
    Because $\delta(G) \geq 2$, there exists another edge $uw \in E(G)$. Hence, as each edge is contained in a block, there exists a block $B' \in \mathcal{B}$ such that $u \in B'$ and $B' \neq B$. It then follows from Proposition~\ref{prop:reduced_block_tree_subtrees} that there exists a block $B_u \in \mathcal{B}$ such that $u \in B_u$ and $B_u$ adjacent to $B$ in $T_{\mathcal{B}}$. Analogously, there exists a block $B_v \in \mathcal{B}$ adjacent to $B$ in $T_{\mathcal{B}}$ such that $v \in B_v$. By the maximality of block $B$, it follows that $v \not \in B_u$ and $u \not \in B_v$. Hence, $B_u \neq B_v$, as desired.
\end{proof}

\begin{corollary}
    \label{col:reduced_block_tree_leaves}
    Let $T_{\mathcal{B}}$ be a reduced block tree of a connected graph $G$ with $\delta(G) \geq 2$. Then each leaf in $T_{\mathcal{B}}$ corresponds to a $2$-connected block in graph $G$ of at least $3$ vertices.
\end{corollary}
\begin{proof}
    By Proposition~\ref{prop:reduced_block_tree_k2}, blocks of size $2$ cannot be leaves in $T_{\mathcal{B}}$. Then, by the definition of blocks, the result follows. 
\end{proof}

\begin{proposition}
    \label{prop:reduced_block_tree_balanced}
    Let $G$ be a connected graph such that $|B| < n/4$ for all blocks $B \in \mathcal{B}(G)$, and let $T_{\mathcal{B}}$ be a corresponding reduced block tree. Then there exists a vertex ${B^*} \in V(T_{\mathcal{B}})$ and a colouring $\phi : V(T_{\mathcal{B}}) \setminus \{B^*\} \to \{\text{red$,$ blue}\}$ such that all components of $T_{\mathcal{B}} - {B^*}$ are monochromatic and that for $S_{\text{red}} = \{v \in B \setminus B^* \, | \, B \in \mathcal{B}, \phi(B) = \text{red}\}$ and $S_{\text{blue}} = \{v \in B \setminus B^* \, | \, B \in \mathcal{B}, \phi(B) = \text{blue}\}$ it holds that $|S_{\text{blue}}| \leq |S_{\text{red}}| \leq 3|S_{\text{blue}}|$.
\end{proposition}
\begin{proof}
    Firstly, we note by Proposition~\ref{prop:reduced_block_tree_subtrees} that $S_{\text{red}} \cap S_{\text{blue}} = \emptyset$ and hence $V(G)$ is partitioned by the sets $B^*$, $S_{\text{red}}$, and $S_{\text{blue}}$. Therefore, $|B^*| + |S_{\text{red}}| + |S_{\text{blue}}| = n$.
    
    Assume that the proposition does not hold. Then, let ${B^*} \in V(T_{\mathcal{B}})$ and $\phi : V(T_{\mathcal{B}}) \to \{\text{red$,$blue}\}$ be a vertex and a colouring respectively such that all components of $T_{\mathcal{B}} - {B^*}$ are monochromatic, $|S_{\text{red}}| \geq |S_{\text{blue}}|$, subject to which $|S_{\text{red}}|$ is minimised. We note that as it concerns a counterexample, we must have $|S_{\text{red}}| > 3|S_{\text{blue}}|$. 
    
    We observe that as $|B^*| < n/4$, $T_{\mathcal{B}} - {B^*}$ is non-empty. Therefore, due to $|S_{\text{red}}| \geq |S_{\text{blue}}|$, $T_{\mathcal{B}} - {B^*}$ contains at least one red component. Suppose that $T_{\mathcal{B}} - {B^*}$ contains exactly one red component. Then, because $T_{\mathcal{B}}$ is a tree, vertex ${B^*}$ has exactly one red neighbour ${B'} \in V(T_{\mathcal{B}})$ in $T_{\mathcal{B}}$. Then consider using vertex ${B'}$ instead of vertex ${B^*}$, uncolouring ${B'}$ and colouring ${B^*}$ blue. Let $\phi'$ denote the resulting new colouring, and let $S_{\text{red}}'$ and $S_{\text{blue}}'$ be the sets of vertices in $G$ corresponding to $\phi'$. We note that as blocks $B^*$ and $B'$ both contain less than $n/4$ vertices, it holds that $|S_{\text{red}}'| > |S_{\text{red}}| - n/4$ and $|S_{\text{blue}}'| < |S_{\text{blue}}| + n/4$. Moreover, we note that by the maximality of blocks, $B^* \setminus B', B' \setminus B^* \neq \emptyset$, and hence $|S_{\text{red}}'| < |S_{\text{red}}|$ and $|S_{\text{blue}}'| > |S_{\text{blue}}|$. If $|S_{\text{red}}'| > |S_{\text{blue}}'|$, the new colouring $\phi'$ is more balanced, and thus contradicts the minimality of $|S_{\text{red}}|$. Therefore, it holds that $|S_{\text{red}}'| < |S_{\text{blue}}'|$. Because we assumed that $|S_{\text{red}}| > 3|S_{\text{blue}}|$, and as $|B^*| + |S_{\text{red}}| + |S_{\text{blue}}| = n$, it follows that $|S_{\text{blue}}| \leq n/4$. Thus, $|S_{\text{blue}}'| < |S_{\text{blue}}| + n/4 \leq n/2$. But then, as $|B'| + |S_{\text{red}}'| + |S_{\text{blue}}'| = n$, it follows that
    \begin{align*}
        |S_{\text{red}}'| &= n - |S_{\text{blue}}'| - |B'|\\
        & > n - \frac{n}{2} - \frac{n}{4}\\
        &= \frac{n}{4}.
    \end{align*}
    Then, inverting the colours red and blue results in a colouring satisfying all the conditions of the proposition, contradicting  $T_{\mathcal{B}}$ being a counterexample.

    Hence, we may assume that forest $T_{\mathcal{B}} - {B^*}$ contains at least $2$ red components. Then let $C_1, C_2, \ldots, C_{\ell}$ be the red components of $T_{\mathcal{B}} - {B^*}$, and let $S_1, S_2, \ldots , S_{\ell}$ be defined by $S_i = \{v \in B \setminus B^* \, | \, B \in C_i\}$ for $i \in [\ell]$. Then, by Proposition~\ref{prop:reduced_block_tree_subtrees}, the sets $S_1, S_2, \ldots , S_{\ell}$ partition set $S_{\text{red}}$.
    
    Suppose that there exists an index $i \in [\ell]$ such that $|S_i| > |S_{\text{blue}}|$. Then, recolouring all blue components red, and recolouring component $C_i$ blue leads to sets $S_{\text{red}}'$ and $S_{\text{blue}}'$ such that, as $\ell \geq 2$, $\min(|S_{\text{red}}'|, |S_{\text{blue}}'|) > \min(|S_{\text{red}}|, |S_{\text{blue}}|)$. Thus, as $|S_{\text{red}}'| + |S_{\text{blue}}'| = |S_{\text{red}}| + |S_{\text{blue}}|$, by possibly inverting the colours, we find a more minimal counterexample. Hence, we may assume that $|S_i| \leq |S_{\text{blue}}|$ for all $i \in [\ell]$. Then, as $|S_{\text{red}}| = \sum_{i=1}^{\ell} |S_i|$, we find that $|S_{\text{red}}| \leq \ell |S_{\text{blue}}|$. Therefore, as $|S_{\text{red}}| > 3|S_{\text{blue}}|$, it holds that $\ell > 3$.

    Similarly, suppose that there exists an index $i \in [\ell]$ such that $|S_i| < (|S_{\text{red}}| - |S_{\text{blue}}|) /2$. Then clearly recolouring component $C_i$ blue contradicts the minimality of $|S_{\text{red}}|$. Hence, we may assume that $|S_i| \geq (|S_{\text{red}}| - |S_{\text{blue}}|) /2$ for all $i \in [\ell]$. Then, as $|S_{\text{red}}| = \sum_{i=1}^{\ell} |S_i|$, we find that $|S_{\text{red}}| \geq \ell \cdot (|S_{\text{red}}| - |S_{\text{blue}}|) /2$. It then follows that, because $\ell > 3$, $|S_{\text{red}}| \leq \frac{\ell}{\ell - 2} |S_{\text{blue}}|$. But then, as $\frac{\ell}{\ell - 2} < 3$ for $\ell > 3$, we conclude that vertex ${B^*}$ and colouring $\phi$ do not form a counterexample. Thus, we conclude that the proposition holds.
\end{proof}

\subsubsection{Building $2$-connected semi-random graphs}
\label{subsection:proof_upper_bound_k-connectedness}
In this section, we describe our strategy and analyse the corresponding process for building a $2$-connected semi-random graph, and obtain the tight upper bound of $\tau_{\mathcal{C}_2}$ in the pre-positional process as in Lemma~\ref{thm:2-connectedness}. Our strategy consists of three phases.

In the first phase, we use the $2$-min process as described in Section~\ref{subsection:min-degree_process}. The following proposition shows useful properties of the resulting graph.

\begin{proposition} 
    \label{prop:2-connectedness_small_blocks}
    Let $G$ be the semi-random graph resulting from the $2$-min process. Then, a.a.s., $G$ contains $o(n)$ vertices that are contained in $2$-connected induced subgraphs of order at most $\sqrt{\ln n}$ in graph $G$.
\end{proposition}
\begin{proof}

Let $X$ be the number of vertices contained in $2$-connected induced subgraphs of order at most $\sqrt{\ln n}$.
We note that it suffices to show that $\ex{X}=o(n)$. Moreover, let $Y_\ell$ denote the number of $2$-connected induced subgraphs of order $\ell$ for $1 \leq \ell \leq \sqrt{\ln n}$. Thus, by linearity of expectation, $\mathbb{E}[X] \leq \sum_{\ell = 1}^{\sqrt{\ln n}} \ell \mathbb{E}[Y_\ell]$.
    
For $1 \leq \ell \leq \sqrt{\ln n}$, let $Z_\ell$ denote the number of induced subgraphs of order $\ell$ with at least $\ell$ edges. Because each $2$-connected graph contains at least as many edges as the vertices, it follows immediately that $Y_\ell \leq Z_\ell$, and thus,  $\ex{X} \leq \sum_{\ell = 1}^{\sqrt{\ln n}} \ex{\ell Z_\ell}$. Hence it suffices to show that $\sum_{\ell = 1}^{\sqrt{\ln n}} \mathbb{E}[\ell Z_\ell] = o(n)$.

Let $1\leq\ell\leq \sqrt{\ln n}$, and fix $S\subseteq [n]$ such that $|S|=\ell$. Let $p_S$ be the probability that $G[S]$ contains at least $\ell$ edges. Note that $\ex{Z_{\ell}}=\sum_{S\in \binom{[n]}{\ell}}p_S$. Next, we estimate $p_S$.
    
We first split the 2-min process into two phases. The first phase ends after the step where the last isolated vertex becomes incident with an edge, and thus the second phase starts with a graph with minimum degree one. We further split each phase into subphases for analysis. Specifically, for the first phase we define subphases $\alpha_1, \alpha_2, \ldots$ such that $\alpha_i$ consists of the steps where $\frac{n}{2^i} < |\{v \in V(G) \, | \, \deg(v) = 0\}| \leq \frac{n}{2^{i-1}}$ for $i \in \{1, 2, \ldots \}$.  We note that these subphases are well defined, as by the definition of the first phase of the $2$-min process, the number of isolated vertices is strictly decreasing. We then define subphases $\beta_1, \beta_2, \ldots$ of the second phase the $2$-min process such that subphase $\beta_i$ consists of the steps where $\frac{n}{2^i} < |\{v \in V(G) \, |  \, \deg(v) = 1\}| \leq \frac{n}{2^{i-1}}$ for $i \in \{1, 2, \ldots \}$. Note that some of the subphases might be empty, e.g.\ subphase $\beta_1$ is empty if the number of vertices with degree $1$ at the beginning of the second phase is already smaller than $n/2$. We observe that there are $\log_2 n$  subphases of both phases of the $2$-min process.

To bound $p_S$, we first choose a set $T$ of $\ell$ edges from $\binom{S}{2}$. There are thus $\binom{{\binom{\ell}{2}}}{\ell}\le \binom{\ell^2}{\ell}$ choices for set $T$. Then we determine an ordering for the edges in $T$. There are $\ell!$ ways to fix such an ordering. Fixing an ordering $e_1,\ldots, e_{\ell}$, we bound the probability that these edges are added to $G$ in this order. The probability that a specific edge $xy\in T$ is added in a specified step in subphase $\alpha_i$ (and $\beta_i$) is at most $2\cdot \frac{2^{i-1}}{n}\cdot\frac{1}{n} = \frac{2^{i}}{n^2}$, since the first vertex of the edge is chosen u.a.r.\ from the isolated vertices, of which there are at most $n/2^{i-1}$, and the second vertex is chosen u.a.r.\ from all vertices. The factor $2$ accounts for whether $x$ or $y$ is the square or the circle of the edge (note that due to the structure of the $2$-min process, sometimes only one of the two may be relevant).
    
Let $\ell_{\alpha_i}$ and $\ell_{\beta_i}$ be the number of edges of $e_1,\ldots, e_{\ell}$ that are added in subphases $\alpha_i$ and $\beta_i$ respectively. Let ${\boldsymbol \ell_{\alpha}}=(\ell_{\alpha_i})_{i\ge 0}$ and ${\boldsymbol \ell_{\beta}}=(\ell_{\beta_i})_{i\ge 0}$. Note that the number of isolated vertices decreases by at least 1 in each step of the first phase of the $2$-min process. Thus the number of steps in subphase $\alpha_i$ is at most $\frac{n}{2^{i-1}}-\frac{n}{2^i} = \frac{n}{2^i}$. Thus, given ${\boldsymbol \ell_{\alpha}}$ and ${\boldsymbol \ell_{\beta}}$, there are at most $\prod_i \binom{n/2^i}{\ell_{\alpha_i}}\binom{n/2^i}{\ell_{\beta_i}}$ ways to specify steps in the 2-min process where edges in $T$ are added. Combining all, we have the following bound on $p_S$:
    
\[p_S \leq \binom{\ell^2}{\ell} \ell!\sum_{{\boldsymbol \ell_{\alpha}},{\boldsymbol \ell_{\beta}} } \left(\prod_{i = 1}^{\log_2 n} \binom{n/2^i}{\ell_{\alpha_i}}\binom{n/2^i}{\ell_{\beta_i}} \left( \frac{2^{i}}{n^2} \right)^{\ell_{\alpha_{i}}+\ell_{\beta_{i}}}  \right), \]
where the first summation is over all choices for ${\boldsymbol \ell_{\alpha}}$ and ${\boldsymbol \ell_{\beta}}$ such that $\sum_{i=1}^{\log_2 n} \left(\ell_{\alpha_i}+\ell_{\beta_i}\right)= \ell$.

Using $\binom{n/2^i}{\ell_{\alpha_i}}\leq (n/2^i)^{\ell_{\alpha_i}}$ and $\binom{n/2^i}{\ell_{\beta_i}}\leq (n/2^i)^{\ell_{\beta_i}}$, we then obtain
\[p_S\leq \binom{\ell^2}{\ell} \ell ! n^{-\ell} \sum_{{\boldsymbol \ell_{\alpha}},{\boldsymbol \ell_{\beta}}} 1.\]

The set of $\{({\boldsymbol \ell_{\alpha}},{\boldsymbol \ell_{\beta}})\, | \, \sum_{i=1}^{\log_2 n} \left(\ell_{\alpha_i}+\ell_{\beta_i}\right)= \ell\}$ corresponds to the set of weak integer compositions of $\ell$ into $2\log_2 n$ parts of non-negative integers, and thus has cardinality $\binom{\ell + 2\log_2 n -1}{2\log_2 n-1} \le \binom{\ell + 2\log_2 n }{\ell}$.

Hence, it follows that
	\begin{align*}
		\mathbb{E}[X] &\leq \sum_{\ell = 1}^{\sqrt{\ln n}} \mathbb{E}[\ell Z_{\ell}]\\
		&= \sum_{\ell = 1}^{\sqrt{\ln n}} \left(\ell \cdot \sum_{S \in \binom{[n]}{\ell}} p_S \right)\\
		&\leq \sum_{\ell = 1}^{\sqrt{\ln n}} \ell \binom{n}{\ell} \binom{\ell^2}{\ell} \ell ! n^{-\ell} \binom{\ell + 2\log_2 n}{\ell}.
	\end{align*}

Using $\binom{n}{\ell}\le n^{\ell}/\ell!$, $\binom{\ell^2}{\ell}\le (e\ell)^{\ell}$ and $\binom{\ell + 2\log_2 n}{\ell} \le (e(\ell+\log_2 n)/\ell)^{\ell} \le (10\log_2 n/\ell)^{\ell}  $ (as $\ell\le \sqrt{\ln n}$), we then obtain
\[
\ex{X}\le \sum_{\ell = 1}^{\sqrt{\ln n}}  \ell (10 e \log_2 n)^{\ell} = \exp\left( \sqrt{\ln n}\ln\log_2 n+O(\sqrt{\ln n})\right)=o(n),
\]
as desired.
\end{proof}
\begin{corollary}
    \label{cor:2-connectedness_number_of_components}
    Let $G$ be the semi-random graph resulting from the $2$-min process. Then, a.a.s., $G$ contains $o(n)$ maximal $2$-connected induced subgraphs.
\end{corollary}
\begin{proof}
    Consider the set ${\cal T}:=\{(v,B): v\in B, B\in {\cal B}(G)\}$. Using the block-cut tree structure (Definition~\ref{def:block-cut_tree}), it follows that $|{\cal T}| \leq n + |{\cal B}| - 1$. Moreover, the sets ${\cal T}_B:=\{(v',B')\in T: B'=B\}$ for $B \in \mathcal{B}(G)$ partition ${\cal T}$. Let $B_{1},\ldots, B_{\ell}$ be the set of blocks of size at least $\sqrt{\ln n}$. Then, by Proposition~\ref{prop:2-connectedness_small_blocks}, $\sum_{1\leq i\leq \ell} |{\cal T}_{B_i}| \leq |{\cal T}| \leq n + \ell + o(n)$. However, $|{\cal T}_{B_i}| \geq \sqrt{\ln n}$ for every $i$, and thus it follows then that $\ell \sqrt{\ln n} \leq n + \ell + o(n)$. Thus it follows that $\ell=o(n)$, as desired.
\end{proof}

The resulting graph thus contains $o(n)$ blocks of size at least $3$. Because we have not bounded the number of blocks consisting of $2$ vertices, we will use Corollary~\ref{col:reduced_block_tree_leaves} and the other structural results in Section~\ref{subsection:supporting_structural_results} to ensure the graph becomes $2$-connected.

Let $G_1$ be the graph obtained after the first phase, i.e.\ the graph resulting from the 2-min process. In the second phase, we add semi-random edges to make $G_1$ connected. The following proposition shows that we can achieve this a.a.s.\ with $o(n)$ additional semi-random edges. 
\begin{proposition}
    \label{prop:2-connectedness_connect}
    A.a.s.\ $G_1$ can be made connected by the addition of $o(n)$ semi-random edges.
\end{proposition}
\begin{proof}
    By Corollary~\ref{cor:2-connectedness_number_of_components}, $G_1$ contains $o(n)$ maximal $2$-connected induced subgraphs. We claim that each vertex not contained in a $2$-connected induced subgraph is contained in a component that contains a $2$-connected induced subgraph. Suppose not. Then $G_1$ must contain a tree component, contradicting the fact that the minimum degree of $G_1$ is at least two. Hence the number of components of graph $G_1$ is bounded from above by the number of maximal $2$-connected induced subgraphs, and therefore is $o(n)$.

    By choosing $v_t$ to be one of the vertices in the smallest component, each semi-random edge has a probability of at least $1/2$ to decrease the number of components. Hence, by standard concentration arguments, $G_1$ can be made connected in $o(n)$ additional rounds.
\end{proof}

Let $G_2$ be the graph obtained after the second phase. In the third phase, we ensure that $G_2$ becomes $2$-connected by adding $o(n)$ semi-random edges. 
\begin{proposition}
    \label{prop:2-connectedness}
    A.a.s.\ $G_2$ can be made $2$-connected by the addition of $o(n)$ semi-random edges.
\end{proposition}
\begin{proof}
    Let $\mathcal{B}$ be the block decomposition of $G_2$ and $T_{\mathcal{B}}$ be a reduced block tree of $G_2$. By Corollary~\ref{col:reduced_block_tree_leaves}, each leaf in $T_{\mathcal{B}}$ is a $2$-connected block. Thus, by Corollary~\ref{cor:2-connectedness_number_of_components}, $T_{\mathcal{B}}$ a.a.s.\ contains $o(n)$ leaves.
    
    First consider the case that ${\cal B}$ contains a block $B^*$ such that $|B^*| \geq n / 4$. We consider the following strategy. Take an arbitrary enumeration $B_1,\ldots, B_h$ of all leaf blocks of $T_{\mathcal{B}}$. For each $1\le j\le h$, we will add a semi-random edge between $B_j$ and $B^*$ in increasing order of $j$. Suppose these semi-random edges have already been added between $B_i$ and $B^*$ for all $i<j$. Let $B_jB'_1B'_2\ldots B'_{\ell}B^*$ be the unique path from $B_j$ to $B^*$ in $T_{\mathcal{B}}$. Moreover, let $x$ be the unique vertex in $B_j \cap B'_1$, and $y$ the unique vertex in $B'_{\ell} \cap B^*$. Note that possibly $x = y$. Then, in each subsequent round $t$, we choose $v_t$ to be an arbitrary vertex in $B_j \setminus \{x\}$. If $u_t$ is contained in $B^*\setminus \{y\}$, we add the edge $u_tv_t$. If instead square $u_t$ is not contained in $B^* \setminus \{y\}$, we consider the round a failure.
    
    Note that in each round, the probability of the second vertex landing in $B^*\setminus \{y\}$ is $(|B^*|-1)/n \geq 1/4-o(1)$, and as a.a.s.\ $T_{\mathcal{B}}$ contains $o(n)$ leaves, the number of rounds required to add semi-random edges between $B^*$ and all $B_1,\ldots, B_h$ is $o(n)$ in expectation.
    
    Let $G_2'$ be the graph resulting from the addition of the $h$ semi-random edges as described above. Then, for each leaf block $B$, $G_2'$ contains two vertex-disjoint paths from $B$ to $B^*$. Namely, one path via the blocks on the path between $B$ and $B^*$ in $T_{\mathcal{B}}$, and the other being the edge that was added between $B$ and $B^*$. Because this holds for all leaves, using Proposition~\ref{prop:reduced_block_tree_k2}, the resulting graph is $2$-edge-connected. Moreover, as each block is on a cycle with $B^*$ and a leaf, and as the blocks of size at least $3$ are $2$-connected, for each cut-vertex $v$ it follows that graph $G_2' - v$ contains one large component containing $B^* \setminus \{v\}$, and all other components are of the form $B \setminus \{v\}$ where $B \in \mathcal{B}$ is a block of size at least $3$. We note that these blocks $B$ such that $B \setminus \{v\}$ is a component for some cut-vertex $v \in [n]$ correspond exactly to the blocks that are leaves in the block-cut tree (Definition~\ref{def:block-cut_tree}), but not in $T_{\mathcal{B}}$. 
    
    By argumentation analogous to that used in the proof of Corollary~\ref{col:reduced_block_tree_leaves}, all such blocks $B$ are $2$-connected. Hence, by Proposition~\ref{cor:2-connectedness_number_of_components}, there are $o(n)$ such blocks. Moreover, each such a block contains at most one cut-vertex. We then use the following strategy to absorb these cut-vertices. We iteratively consider pairs $(B, v)$ where $v \in B$ is a cut-vertex and $B \in \mathcal{B}$ a block such that $B-v$ is a component when removing $v$. As noted earlier, there are $o(n)$ such pairs. If $|B \setminus \{v\}| \leq n/2$, we choose $v_t \in B \setminus \{v\}$ arbitrarily. With probability at least $1/2$, $u_t \in [n] \setminus B$. Similarly, if $|B \setminus \{v\}| < n/2$, we choose $v_t \in [n]\setminus B$, and with probability at least $1/2 - o(1)$, $u_t \in B \setminus \{v\}$. In either case, $v$ no longer separates block $B$ from the rest of the graph. Note that as this described the only configuration of remaining cut-vertices in the graph, eliminating all such pairs eliminates all cut-vertices. Since there are $o(n)$ such pairs, the total number of rounds needed to absorb all such cut-vertices is $o(n)$ in expectation.
    
    It thus takes at most $o(n)$ rounds in total in expectation to ensure that the graph becomes $2$-connected. Standard concentration inequalities such as Chernoff bounds then immediately imply that also a.a.s.\ it takes $o(n)$ rounds to extend $G_2$ to a $2$-connected graph. 
    
    Hence we may assume that each block $B$ in $\mathcal{B}$ is of size strictly smaller than $n/4$. We use a different strategy in this case. Instead of adding edges from leaves to a single block, we will consider balancing the tree into two subforests. We will then add edges between the leaves in one forest and vertices in the other forest, and vice versa.
    
    Let vertex ${B^*} \in V(T_{\mathcal{B}})$, colouring $\phi : V(T_{\mathcal{B}}) \setminus \{B^*\} \to \{\text{red$,$ blue}\}$, and sets $S_{\text{red}}$ and $S_{\text{blue}}$ be as given by Proposition~\ref{prop:reduced_block_tree_balanced}. For each $v \in B^*$ let $T_{\mathcal{B}, v}$ denote the components of $T_{\mathcal{B}} - v$ that contain a block containing $v$. Thus, $T_{\mathcal{B}, v}$ denotes the blocks $B$ where $v$ is the last cut-vertex on the path from $B$ to $B^*$ in $T_{\mathcal{B}}$. We refer to $T_{\mathcal{B}, v}$ as the branch rooted at $v$. Moreover, let $S_{\mathcal{B}, v}$ denote $\bigcup_{B \in V(T_{\mathcal{B}, v})} B \setminus B^*$. That is, $S_{\mathcal{B}, v}$ is the set of all vertices contained in blocks in $T_{\mathcal{B}, v}$ except for vertex $v$ itself. If $|S_{\mathcal{B}, v}| \leq n/8$, we say branch $T_{\mathcal{B}, v}$ is small. Otherwise we say $T_{\mathcal{B}, v}$ is big. Finally, for all leaf blocks $B$, let $v_{B}$ denote the vertex that block $B$ has in common with the next block on the path from $B$ to $B^*$ in $T_{\mathcal{B}}$.
    
    We first consider the leaves of $T_{\mathcal{B}}$ contained in small branches. Take two arbitrary enumerations $B_1, B_2, \ldots, B_{h_1}$ and $R_1, R_2, \ldots, R_{h_2}$ of all blue and red leaf blocks of $T_{\mathcal{B}}$ contained in small branches respectively. We will iteratively add edges between $B_j$ and $S_{\text{red}}$ in increasing order of $j$, and analogously between $R_j$ and $S_{\text{blue}}$. Suppose that semi-random edges have already been added between $B_i$ and $S_{\text{red}}$ for all $i < j$. Let $T_{\mathcal{B}, v}$ be the branch containing leaf $B_j$. We then choose $v_t$ to be an arbitrary vertex in $B_j \setminus \{v_{B_j}\}$. Because $|B_j| \geq 2$, such a choice for $v_t$ always exists. Then, if $u_t$ lands in $S_{\text{red}} \setminus S_{\mathcal{B}, v}$, we add edge $u_tv_t$. Otherwise we consider the round a failure.
    
    Analogously, for $R_j$ the red leaf in a small branch $T_{\mathcal{B}, v}$ with the lowest index that has not previously received a circle, we choose $v_t$ in $R_j \setminus \{v_{R_j}\}$. If $u_t$ is contained in $S_{\text{blue}} \setminus S_{\mathcal{B}, v}$, we add the edge $u_tv_t$, and otherwise we consider the round a failure.
    
    Then, as tree $T_{\mathcal{B}}$ has $o(n)$ leaves, there are $o(n)$ blue and $o(n)$ red leaves. Moreover, by Proposition~\ref{prop:reduced_block_tree_balanced} $|S_{\text{blue}}| \leq |S_{\text{red}}| \leq 3|S_{\text{blue}}|$, and $|S_{\text{red}}| + |S_{\text{blue}}| \geq 3n/4$. Thus, the probability that a vertex from $S_{\text{red}} \setminus S_{\mathcal{B}, v}$ is chosen u.a.r.\ where $T_{\mathcal{B}, v}$ is a small branch, is at least $3n/8 - n/8 = n/4$. Similarly, the probability that a vertex from $S_{\text{blue}} \setminus S_{\mathcal{B}, v}$ is chosen u.a.r.\ where $T_{\mathcal{B}, v}$ a small branch, is at least $3n/16 - n/8 = n/16$. Hence, the expected number of rounds needed to add edges to all leaf blocks in small branches is $o(n)$.
    
    Next, we consider the leaf blocks in big branches. We first note that there are at most $7$ big branches. We use a similar strategy as for the small branches, but drop the requirement that $u_t$ and $v_t$ must be in distinct branches. Again take two arbitrary enumerations $B_1, B_2, \ldots, B_{h_3}$ and $R_1, R_2, \ldots, R_{h_4}$ of all blue and red leaf blocks of $T_{\mathcal{B}}$ contained in big branches respectively. Suppose that semi-random edges have already been added between $B_i$ and $S_{\text{red}}$ for all $i < j$. We then choose $v_t$ to be an arbitrary vertex in $B_j \setminus \{v_{B_j}\}$. Because $|B_j| \geq 2$, such a choice for $v_t$ always exists. If $u_t$ lands in $S_{\text{red}}$, we add edge $u_tv_t$. Otherwise, we consider the round a failure. The strategy for red leaf blocks is analogous.
    
    Because the probability that a vertex from $S_{\text{red}}$ is chosen u.a.r.\ is at least $3/8$, and the probability that a vertex from $S_{\text{blue}}$ is chosen u.a.r.\ is at least $3/16$, it also takes $o(n)$ rounds in expectation to add edges to all leaf blocks in big branches. 
    
    After all leaves in both small and big branches have received an edge, there exist two internally vertex-disjoint paths from each leaf block $B$ to $B^*$. Namely, as all of the edges we added have one red and one blue endpoint, each blue leaf has a path which only contains blue vertices and a vertex in $B^*$, and a path that starts with the added edge, and then only contains red vertices and one vertex in $B^*$. Analogously, there exist two such paths from each red leaf. As these two paths do not share their endpoint in leaf $B$, and as each leaf is $2$-connected by Corollary~\ref{col:reduced_block_tree_leaves}, set $B \setminus \{v_B\}$ does not contain any cut-vertices.
    
    We note that again the resulting graph is $2$-edge-connected. We then use the same strategy as in the case where there exists a block of size at least $n/4$ to eliminate all the cut-vertices that separate individual blocks from the rest of the graph. Recall that this strategy a.a.s.\ takes $o(n)$ rounds. Let $G_2''$ be the resulting graph. We then observe that no vertex in $[n] \setminus B^*$  is a cut-vertex in graph $G_2''$. Hence, we consider a cut-vertex $v \in B^*$. First suppose that the branch rooted at $v$ is empty. We observe that by Proposition~\ref{prop:reduced_block_tree_k2} it then holds that $|B^*| \geq 3$. But then, $B^*$ is $2$-connected, contradicting $v$ being a cut-vertex. Next suppose that the branch rooted at $v$ is small. We note that for each vertex in $S_{\mathcal{B}, v}$ there exists a path to a leaf of branch $T_{\mathcal{B}, v}$ contained in $S_{\mathcal{B}, v}$. As each such a leaf has an edge to another branch, and as $B^*$ is either $2$-connected or isomorphic to $K_2$, it follows that $G_2''-v$ is connected. Hence, $v$ is not a cut-vertex. 
    
    Finally, suppose that the branch rooted at $v$ is big. In this case $v$ may indeed be a cut-vertex. Namely, if $T_{\mathcal{B}, v}$ contains multiple components of different colours, each of the edges added to the leafs in the branch could have both endpoints within the branch. To deal with such cut-vertices, we use a two-step strategy. In the first step, we want to ensure that the subgraph induced by $S_{\mathcal{B}, v}$ becomes connected. We achieve this using the standard strategy of choosing $v_t$ in the smallest component in the subgraph induced by $S_{\mathcal{B}, v}$. If $u_t$ lands in a different component of this subgraph, we add edge $u_tv_t$, otherwise we consider the round a failure. We note that as each component of $T_{\mathcal{B}, v}$ contains at least one leaf of $T_{\mathcal{B}}$, by Corollary~\ref{cor:2-connectedness_number_of_components}, $T_{\mathcal{B}, v}$ contains at most $o(n)$ components. As $|S_{\mathcal{B}, v}| > n/8$, the probability of adding successfully adding an edge in this first step is at least $1/16$. Then, by standard concentration inequalities, this step a.a.s.\ takes $o(n)$ rounds as well. We note that in the resulting graph, cut-vertex $v$ then separates two components, given by vertex sets $S_{\mathcal{B}, v}$ and $[n]\setminus (S_{\mathcal{B}, v} \cup \{v\})$. In the second step of the strategy, we connect these two components by a single edge. By again choosing $v_t$ in the smaller of the two components, and considering the round a failure if $u_t$ does not land in the other component. As the probability of a failure round is thus at most $1/2$, by standard concentration inequalities, the number of rounds in this step is $O(1)$. We then note that there are at most $7$ vertices $v \in B^*$ such $T_{\mathcal{B}, v}$ is big. Hence, the total number of rounds to ensure that each of these cut-vertices is absorbed is a.a.s.\ $o(n)$.
    
    Because the resulting graph then thus no longer contains any cut-vertices, the graph is $2$-connected. Thus, in any case, we can ensure that graph $G_2$ becomes $2$-connected in a.a.s.\ $o(n)$ rounds, as desired.
\end{proof}

Combining the analysis of these individual phases then results in the following lemma.

\begin{lemma}
    \label{lem:2-connectedness_upper_bound}
    $\tau_{\mathcal{C}_2} \leq \ln 2 + \ln(\ln 2 + 1))$ in the pre-positional process.
\end{lemma}
\begin{proof}
    The lemma directly follows from Propositions~\ref{prop:2-connectedness_small_blocks} and~\ref{prop:2-connectedness}, and the fact that the $2$-min process requires $(\ln 2 + \ln(\ln 2 + 1)) + o(1))n$ rounds.
\end{proof}

Lemma~\ref{thm:2-connectedness} then follows directly from Lemmas~\ref{lem:relation_pre_post_positional}, and~\ref{lem:2-connectedness_upper_bound}, and the known lower bound given by the $2$-min process.

This then completes the proof of Theorem~\ref{thm:k-connect}; it follows directly from Lemmas~\ref{lem:pre-positional_k-connectedness},~\ref{lem:pre-positional_1-connectedness}, and~\ref{thm:2-connectedness}.

\section{Degenerate subgraphs: proof of Theorem~\ref{thm:degeneracy}}
\label{section:degeneracy}

%Recall that a graph $H$ is said to be $d$-degenerate if each subgraph of $H$ contains a vertex of degree at most $d$. In their seminal paper, Ben-Eliezer, Hefetz, Kronenberg, Parczyk, Shikhelman, and Stojakovi\'c~\cite{Ben-Eliezer2020Semi-randomProcess} considered the number of rounds needed to construct a fixed size $d$-degenerate graph as a subgraph. They showed the following upper bound in the post-positional model.

%\begin{theorem}[{\cite[Theorem~1.10]{Ben-Eliezer2020Semi-randomProcess}}]
 %   Let $H$ be a fixed $d$-degenerate graph, and let $f: \mathbb{N} \to \mathbb{R}$ be a function such that $\lim_{n \to \infty} f(n) = \infty$. Then there exists a strategy in the post-positional model such that the resulting graph $G$ contains a subgraph isomorphic to $H$ in a.a.s.\ $f(n) \cdot n^{(d-1)/d}$ rounds.
%\end{theorem}

%They conjectured that this bound is tight, which was subsequently shown by Behague, Marbach, Pra{\l}at, and Ruci\'{n}ski~\cite{Behague2021SubgraphHypergraphs}. We show an equivalent upper bound for the pre-positional bound. 

%\begin{theorem}
%\label{thm:degeneracy}
 %   Let $H$ be a fixed $d$-degenerate graph, and let $f: \mathbb{N} \to \mathbb{R}$ be a function such that $\lim_{n \to \infty} f(n) = \infty$. Then there exist a strategy in the pre-positional model such that the resulting graph $G$ contains a subgraph isomorphic to $H$ in a.a.s.\ $f(n)\cdot n^{(d-1)/d}$ rounds.
%\end{theorem}
%\begin{proof}
Part (a) follows by~\cite[Theorem~1.2]{Behague2021SubgraphHypergraphs} and Lemma~\ref{lem:relation_pre_post_positional}.

For part (b), we consider the pre-positional process. Our proof is similar to the proof of Theorem~\ref{thm:degenerate-old}, but requires slightly more careful analysis due to the difference in power between the pre- and post-positional processes.

    Let $g(n)$ be a function such that $g(n)\to \infty$ as $n\to \infty$. We prove that there exist a pre-positional strategy which construct an $H$-subgraph a.a.s.\ in at most $g(n) 2^{|V(H)|}\cdot n^{(d-1)/d}$ rounds. Note that this immediately implies part (b) as $|V(H)|$ is fixed, and we may take $g(n)=f(n) 2^{-|V(H)|}$. We proceed by induction on $|V(H)|$. We note that the statement holds directly if $|V(H)| = 1$. Suppose $H$ is a $d$-degenerate graph with $m\ge 2$ vertices, and assume that the statement holds for all fixed $d$-degenerate graphs $H'$ such that $|V(H')| <m $. 

    Let $v \in V(H)$ such that $\deg_H(v) \leq d$. Consider the graph $H' := H - v$. Then, by the inductive hypothesis, there exists a pre-positional strategy which a.a.s.\ constructs a graph $G'$ containing a copy of $H'$ in at most $T:=g(n) 2^{m-1}\cdot n^{(d-1)/d}$ rounds. Let $C'$ be the copy of $H'$ constructed in $G'$. For each vertex $u \in N_H(v)$, let $u' \in [n]$ be the corresponding vertex in $C'$, and let $N' := \{u' \,: \, u \in N_H(v)\}$. The strategy is then to grow a star from each vertex in $N'$. We do the following subsequently for each $u'\in N'$. Given $u' \in N'$, choose $u_t$ to be $u'$ for $g(n)2^m/(2d) \cdot n^{(d-1)/d}$ subsequent rounds. Let $S_{u'}$ be the set of vertices $w \in [n] \setminus V(C')$ such that $w = v_t$ for at least one of the $g(n)2^m/(2d) \cdot n^{(d-1)/d}$ rounds. Then, by standard concentration arguments, and as $|V(C')|$ is fixed, a.a.s.\ $|S_{u'}|$ is at least $g(n)2^m/(4d) \cdot n^{(d-1)/d}$. Let $G$ be the graph resulting from growing such stars for all $u' \in N'$.

    We then consider the probability that a vertex $w \in [n] \setminus V(C')$ is contained in all such sets, that is $\mathbb{P}\left( w \in \bigcap_{u' \in N'} S_{u'}\right)$. As the construction of $\{S_{u'}\}_{u'\in N'}$ is mutually independent,
    \begin{align*}
        \mathbb{P}\left( w \in \bigcap_{u' \in N'} S_{u'}\right) &\geq \Pi_{u' \in N'} \mathbb{P}\left( w \in S_{u'} \right)\\
        &= \Pi_{u' \in N'} \left(\frac{|S_u|}{n - |V(C')|}\right)\\
        &> \Pi_{u' \in N'} \left(\frac{|S_u|}{n}\right)\\
        &\geq \Pi_{u' \in N'} \left(\frac{g(n)2^m}{4d} \cdot \frac{n^{(d-1)/d}}{n}\right)\\
        &= \Pi_{u' \in N'} \left( \frac{g(n)2^m}{4d} \cdot \frac{1}{n^{1/d}} \right)\\
        &\geq \left( \frac{g(n)2^m}{4d} \cdot \frac{1}{n^{1/d}} \right)^d\\
        &= \left(\frac{g(n)2^m}{4d}\right)^d \cdot \frac{1}{n}.
    \end{align*}
    Let $X := \left| \bigcap_{u' \in N'} S_{u'}\right|$ be a random variable. Then, $\mathbb{E}[X] \geq (g(n)2^m/4d)^d \cdot (n - |V(C')|)/n$, and hence, by standard concentration arguments, as $\lim_{n\to\infty}g(n) = \infty$, a.a.s.\ $\bigcap_{u' \in N'} S_{u'}$ is non-empty. Let $z \in \bigcap_{u' \in N'} S_{u'}$. Consider the subgraph of $G$ given by extending $C'$ with vertex $z$ and the edges between $z$ and $N'$; this subgraph is isomorphic to $H$, as desired. Moreover, the number of rounds to construct $H$ is bounded by up to $g(n)2^{m-1}\cdot n^{(d-1)/d}$ rounds to construct $H'$, together with up to $g(n)2^m/(2d) \cdot n^{(d-1)/d}$ rounds to grow each of the stars. Thus, in total the construction of $H$ requires a.a.s.\ up to
    \begin{align*}
    g(n)2^{m-1}\cdot n^{(d-1)/d} + |N'| \cdot \frac{g(n)2^m}{2d} \cdot n^{(d-1)/d} &\leq \frac{g(n)2^m}{2}\cdot n^{(d-1)/d} + d \cdot \frac{g(n)2^m}{2d} \cdot n^{(d-1)/d}\\ 
    &= g(n)2^m \cdot n^{(d-1)/d}
    \end{align*}
    rounds as desired. \qed

\section{Dense bipartite subgraphs: proof of Theorem~\ref{thm:bipartite}}
\label{section:dense_bipartite}

THe lower bound is trivial, as constructing any such subgraph requires $m$ edges.
For the upper bound, by Lemma~\ref{lem:relation_pre_post_positional} it suffices to consider the pre-positional process. 

    Let $A := [\lceil \sqrt{m} \rceil]$  and $B=[n] \setminus A$. We construct a simple bipartite subgraph with bipartition $(A, B)$ with at least $m$ edges. For a vertex $v \in A$, let $\deg_{G_t}(v, B)$ denote the number of distinct neighbours of $v$ in $B$.

    Our strategy consists of $\lceil \sqrt{m} \rceil$ phases. In the $i^{\text{th}}$ phase, we consistently choose $v_t$ to be vertex $i$. The phase terminates once $\deg_{G_t}(i, B) = \lceil \sqrt{m} \rceil$. We consider a round a failure if $u_t \in A \cup N_{G_{t-1}}(i)$. We observe that the probability of such a failure round is at most $2\lceil\sqrt{m}\rceil / n$. Moreover, because $m = o(n^2)$, we observe that this probability is $o(1)$. 

    Moreover, once all phases have terminated, we note that the bipartition $(A,B)$ forms a bipartite subgraph with at least $\lceil \sqrt{m} \rceil \cdot \lceil \sqrt{m} \rceil \geq m$ non-parallel edges, as desired. Since each round has a probability of being a failure round of $o(1)$, and as there are $\lceil \sqrt{m} \rceil$ phases, each of which requires $\lceil \sqrt{m} \rceil$ successful rounds, the total needed number of rounds is a.a.s.\ $(1 + o(1)) \cdot \lceil \sqrt{m} \rceil \cdot \lceil \sqrt{m} \rceil = (1 + o(1)) m$, following a standard concentration argument. \qed

\section{Large induced cycles: proof of Theorem~\ref{thm:induced}}
\label{section:induced_cycles}

%Next, we discuss a setting which shows the advantage provided by the additional information available in the post-positional model compared to the pre-positional model. 
%Let $\mathcal{C}_{n-1}$ be the property of containing an induced subgraph isomorphic to the simple $(n-1)$-cycle $C_{n-1}$. Note that this property is not monotonically increasing. 
Recall that the semi-random graph processes allow the creation of multi-edges. This will be useful to construct an induced $(n-1)$-cycle in the post-positional process.

%First, we show that achieving the property $\mathcal{C}_{n-1}$ is not difficult in the post-positional model. We provide a simple (likely sub-optimal) strategy to explicitly construct a specified induced $(n-1)$-cycle.

\begin{lemma}
    There exists a post-positional strategy that constructs an induced $(n-1)$-cycle a.a.s.\ in  $O(n\ln n)$ rounds. %for any function $f: \mathbb{N} \to \mathbb{R}$ such that $\lim_{n \to \infty} f(n) = \infty$. 
\end{lemma}
% \jc{Simplified the statement; it is only for expositional purpose and I do not think a lengthy statement is worthwhile.}
\begin{proof}
    Our strategy aims to construct an induced cycle on the vertices $\{1, 2, \ldots, n-1\}$. We designate vertex $n$ as the \emph{sink vertex}. That is, if we cannot add a useful edge given $u_t$, we choose $v_t$ to be $n$. Hence, by removing vertex $n$, we obtain a graph which only contains desired edges. 

    The first time that $u_t$ lands on a vertex $v \in [n-1]$, we choose $v_t$ to be $v + 1$ (unless $v = n-1$, then we choose $u_t$ to be $1$). Any subsequent time $u_t$ lands on $v$, we choose $v_t$ to be $n$. Note that once we have landed at least once on each vertex in $[n-1]$, we have constructed an induced spanning cycle on the set $[n-1]$, as desired.

    Hence, an induced $(n-1)$-cycle is constructed once each vertex in $[n-1]$ is hit at least once, and this takes a.a.s.\ $O(n\log n)$ steps by the coupon collector's problem (see~\cite[Theorem~5.13]{mitzenmacher2017probability}).
\end{proof}

We complete the proof of Theorem~\ref{thm:induced} by showing that a.a.s.\ no pre-positional strategy can construct an induced $(n-1)$-cycle. To obtain an induced $(n-1)$-cycle, we first need to construct an induced path on $n-1$ vertices. Suppose that one has constructed such an induced path $P$, which includes all vertices other than $w \in [n]$, after step $t_0$. By the definition of an induced path, $[n]-w$ induces exactly $n-2$ edges, which form an $(n-1)$-path.

\begin{claim}\label{claim1}
    A.a.s.\ $w$ has $\Theta(n)$ distinct neighbours in $[n]-w$.
\end{claim}

\proof $G_{t_0}$ contains at least $n-1$ edges and thus $t_0\ge n-1$. The distribution of the $n-1$ squares in the first $n-1$ steps is the same as is that of uniformly throwing $n-1$ balls into $n$ bins. By the standard Poisson approximation argument, the number of vertices receiving at least three squares is a.a.s.\ $\Theta(n)$. These vertices must all be adjacent to $w$ since $[n]-w$ induces an $(n-1)$-path.\qed  

It follows immediately that a.a.s.\ the only possible induced $(n-1)$-cycle that can be constructed is on $[n]-w$.
Observe that the only way to construct an induced $(n-1)$-cycle on $[n]-w$ is that
    \begin{enumerate}
        \item[(a)] $\{u_t,v_t\}= \{u,v\}$ for some $t\ge t_0+1$, where $u$ and $v$ are the two ends of $P$;
        \item[(b)] for all $t_0<s<t$, $\{u_s,v_s\}\neq \{u,v\}$; and
        \item[(c)] for all $t_0<s<t$, if $v_s\neq w$ then $u_t$ must be $w$. 
    \end{enumerate}
Considering the first step $t$ after $t_0$ that $v_t\neq w$. If $v_t\notin \{u,v\}$ then by (c), $u_t$ must be $w$, which occurs with probability $1/n$. If $v_t\in \{u,t\}$ then by (a,b), $u_t$ must be either $\{u,v\}\setminus \{v_t\}$ or $w$. The probability of this is $2/n$. Hence, the probability that $P$ can be completed into an induced $(n-1)$-cycle is $O(1/n)$. Hence, there does not exist a strategy that a.a.s.\ constructs an induced $(n-1)$-cycle in the pre-positional process, as desired. \qed

\newpage

\bibliographystyle{acm}
\bibliography{main.bib}

\end{document}